\documentclass[11pt]{elsarticle}
\usepackage[left=3cm,bottom=3cm,top = 2cm, right =3cm]{geometry}
\usepackage{graphicx}
\usepackage{amscd,amsmath,amstext,amsfonts,amsbsy,amssymb,amsthm,eufrak,mathrsfs}
\usepackage{hyperref}
\usepackage[toc,page]{appendix}
\usepackage{chngcntr}

\newtheorem{theo}{Theorem}
\newtheorem{lemm}{Lemma}[section]
\newtheorem{coro}[lemm]{Corollary}
\newtheorem{prop}[lemm]{Proposition}

\newtheorem{defi}[lemm]{Definition}
\newtheorem{rema}[lemm]{Remark}

\newcommand{\mL}{\mathcal{L}}
\newcommand{\K}{\mathcal{K}}

\newcommand{\T}{\mathscr{T}}
\newcommand{\R}{\mathbb{R}}
\newcommand{\RN}{\mathbb{R}^N}

\makeatletter
\def\ps@pprintTitle{
 \let\@oddhead\@empty
 \let\@evenhead\@empty
 \def\@oddfoot{}%
 \let\@evenfoot\let\@oddfoot }
\makeatother
\title{\Large \textbf{{Finite degrees of freedom for the refined blow-up profile of the  semilinear heat equation}}}
\author{V. T. Nguyen\\
\textit{\small{Department of Mathematics, New York University in Abu Dhabi,\\ Saadiyat Island, Abu Dhabi, United Arab Emirates.}}\\
 H. Zaag \footnote{This author is supported by the ERC Advanced Grant no. 291214, BLOWDISOL and by the ANR project ANA\'E ref. ANR-13-BS01-0010-03.\\ -----------------\\ \today} \\ \textit{\small{Universit{\'e} Paris 13, Sorbonne Paris Cit{\'e},\\
LAGA, CNRS (UMR 7539), F-93430, Villetaneuse, France.}}}

\begin{document}
\begin{abstract} 
We refine the asymptotic behavior of solutions to the semilinear heat equation with Sobolev subcritical power nonlinearity which blow up in some finite time at a blow-up point where the (supposed to be generic) profile  holds. In order to obtain this refinement, we have to abandon the explicit profile function as a first order approximation, and take a non explicit function as a first order description of the singular behavior. This non explicit function is in fact a special solution which we construct, obeying some refined prescribed behavior. The construction relies on the reduction of the problem to a finite dimensional one and the use of a topological argument based on index theory to conclude. Surprisingly, the new non explicit profiles which we construct make a family with finite degrees of freedom, namely $\frac{(N+1)N}{2}$ if $N$ is the dimension of the space.\\
 

\end{abstract}

\maketitle
\section{Introduction.}
We are interested in the following semilinear heat equation: 
\begin{equation}\label{equ:problem}
\left\{
\begin{array}{rcl}
u_t &=& \Delta u + |u|^{p-1}u, \\
u(0) &=& u_0 \in L^\infty(\mathbb{R}^N),
\end{array}
\right.
\end{equation}
where $u(t): x \in \mathbb{R}^N \to u(x,t) \in \mathbb{R}$, $\Delta$ denotes the Laplacian in $\mathbb{R}^N$, and 
$$p > 1\quad \text{or} \quad 1 < p < \frac{N + 2}{N-2}\quad  \text{if}\quad  N \geq 3.$$

Equation \eqref{equ:problem} is a simple model for a large class of nonlinear parabolic equations. In fact, it captures features common to a whole range of blow-up problems arsing in various physical situations, particularly it highlights the role of scaling and self-similarity. Among related equations, we would like nonetheless to mention: the solid fuel ignition model (Bebernes, Bressan and Eberly \cite{BBEiumj87}), the thermal explosion (Bebernes and Kassoy \cite{BKsiam81}, Kassoy and Poland \cite{KPsiam80}), \cite{KPsiam81}), surface diffusion (Bernoff, Bertozzi and Witelski \cite{BBAWjsp98}), the motion by mean curvature (Soner and Souganidis \cite{SScpde93}), vortex dynamics in superconductors (Chapman, Hunton and Ockendon \cite{CHOqam98}, Merle and Zaag \cite{MZnon97}).  \\

By standard results, the problem \eqref{equ:problem} has a unique classical solution $u(x,t)$ continuous in time with values in $L^\infty(\mathbb{R}^N)$, which exists at least for small times. The solution $u(x,t)$ may develop singularities in some finite time (see  Kaplan \cite{Kapcpam63}, Fujita \cite{FUJsut66},  Levine \cite{LEVarma73}, Ball \cite{BALjmo77}, Weissler \cite{WEIjde84} for the existence of finite-time blow-up solutions to \eqref{equ:problem}). In this case, we say that $u(x,t)$ blows up in a finite time $T < +\infty$ in the sense that 
$$\lim_{t \to T} \|u(t)\|_{L^\infty(\mathbb{R}^N)} = +\infty.$$
Here we call $T$ the blow-up time of $u(x,t)$. In such a blow-up case, we say that $\hat{a} \in \RN$ is a blow-up point of $u$ if $u$ is not locally bounded in the neighborhood of $(\hat{a},T)$, this means that there exists $(x_n,t_n) \to (\hat{a},T)$ such that $|u(x_n,t_n)| \to +\infty$ when $n \to +\infty$.\\

Let us consider $u(t)$ a solution of \eqref{equ:problem} which blows up in finite time $T$ at only one blow-up point $\hat{a}$. From the translation invariance of \eqref{equ:problem}, we may assume that $\hat{a} = 0$.  Studying the solution $u(x,t)$ near the singularity $(0, T)$ is based on the following \textit{similarity variables} (see \cite{GKiumj87, GKcpam89}):
\begin{equation}\label{def:simivars}
\T[u](y,s) = (T-t)^\frac{1}{p-1}u(x,t), \quad y = \frac{x}{\sqrt{T - t}}, \quad s = -\log(T - t),
\end{equation}
and $w = \T[u]$ solves a new parabolic equation in $(y,s)$, 
\begin{equation}\label{equ:w}
\partial_s w = \mL w - \frac{p}{p-1}w + |w|^{p-1}w, \quad (y,s) \in \RN \times [-\log T, +\infty),
\end{equation}
where 
\begin{equation}\label{def:opL}
\mL = \Delta  - \frac{y}{2}\cdot \nabla + 1.
\end{equation}
\noindent In view of \eqref{def:simivars}, the study of $u(x,t)$ as $(x, t) \to (0, T)$ is then equivalent to the study of $\T[u](y,s)$ as $s \to +\infty$, and each result for $u$ has an equivalent formulation in term of $\T[u]$. \\

According to Giga and Kohn in \cite{GKcpam89} (see also \cite{GKcpam85, GKiumj87}), we know that:

\medskip

\noindent\textit{If $\hat{a}$ is a blow-up point of $u$, then
\begin{equation}\label{equ:limitw}
\lim_{t \to T} (T-t)^\frac{1}{p-1}u(\hat{a} + y\sqrt{T-t},t) = \lim_{s \to +\infty} \T[u](y,s) = \pm \kappa,
\end{equation}
uniformly on compact sets $|y| \leq R$, where $\kappa = (p-1)^{-\frac{1}{p-1}}$.}

\medskip

The estimate \eqref{equ:limitw} has been refined until the higher order by Filippas, Kohn and Liu \cite{FKcpam92}, \cite{FLaihn93}, Herrero and Vel\'azquez \cite{HVdie92}, \cite{HVaihn93},  \cite{VELcpde92}, \cite{VELiumj93}, \cite{VELtams93}. More precisely, they classified the behavior of $\T[u](y,s)$ for $|y|$ bounded, and showed that one of the following cases occurs (up to replacing $u$ by $-u$ if necessary),

\medskip

\noindent - \textit{Case 1 (non-degenerate rate of blow-up): There exists $\ell \in \{1, \cdots, N\}$, and up to an orthogonal transformation of space coordinates,
\begin{equation}\label{equ:c1clas}
\forall R > 0, \; \sup_{|y| \leq R}\left| \T[u](y,s) - \left[\kappa +\frac{\kappa}{4ps} \left(2\ell  - \sum_{i = 1}^\ell |y_i|^2\right)\right] \right| = \mathcal{O}\left(\frac{\log s}{s^2}\right).
\end{equation}}
- \textit{Case 2 (degenerate rate of blow-up): There exists $\mu > 0$ such that 
\begin{equation}
\forall R > 0, \; \sup_{|y| \leq R} \left|\T[u](y,s) - \kappa \right| = \mathcal{O}(e^{-\mu s}),
\end{equation}
(this exponential convergence has been refined up to the order 1 by Herrero and Vel\'azquez, but we omit that description since we choose in this work to concentrate on the non-degenerate rate of blow-up mentioned in the case 1 above).}

\medskip

If $\ell = N$, then $\hat{a} = 0$ is an isolated blow-up point from Vel\'azquez \cite{VELcpde92}. Merle and Zaag \cite{MZcpam98, MZgfa98, MZma00} (with no sign condition), and Herrero and Vel\'azquez \cite{VELcpde92, HVaihn93} (in the positive case) established the following blow-up profile in the variable $\xi = \frac{y}{\sqrt{s}}$ (which is the intermediate scale that separates the regular and singular parts in the non-degenerate case):
\begin{equation}\label{equ:c1prof}
\forall R>0,\quad \sup_{|\xi| \leq R} \left|\T[u](\xi\sqrt{s},s) - f(\xi)\right| \to 0 \quad \text{as $s \to +\infty$},
\end{equation}
where 
\begin{equation}\label{def:fk}
f(\xi) = \kappa\left(1 + \frac{p-1}{4p}|\xi|^2\right)^{-\frac{1}{p-1}}.
\end{equation}
Herrero and Vel\'azquez \cite{HVasnsp92} proved that the profile \eqref{def:fk} is generic in the case $N = 1$, and they announced the same for $N \geq 2$, but they never published it. 

Merle and Zaag \cite{MZcpam98}, \cite{MZgfa98}, \cite{MZma00} derived   the limiting profile in the $u(x,t)$ variable, in sense that $u(x,t) \to u^*(x)$ when $t \to T$ if $x \neq 0$ and $x$ is the neighborhood of $0$, with 
\begin{equation}\label{equ:c1limit}
u^*(x)\sim \left[\frac{8p|\log |x||}{(p-1)^2|x|^2} \right]^\frac{1}{p-1} \quad \text{as} \; x \to 0.
\end{equation}
They also showed that all the behaviors \eqref{equ:c1clas} with $\ell = N$, \eqref{equ:c1prof} and \eqref{equ:c1limit} are equivalent. 

Bricmont and Kupiainen \cite{BKnon94}, Merle and Zaag in \cite{MZdm97} showed the existence of initial data for \eqref{equ:problem} such that the corresponding solutions blow up in finite time $T$ at only one blow-up point $\hat{a} = 0$ and verify the behavior \eqref{equ:c1prof}. Note that the method of \cite{MZdm97} allows to derive the stability of the blow-up behavior \eqref{equ:c1prof} with respect to perturbations in the initial data or the nonlinearity (see also Fermanian, Merle and Zaag \cite{FMZma00}, \cite{FZnon00} for other proofs of the stability).\\

In this work, considering the expansions \eqref{equ:c1clas} with $\ell = N$, \eqref{equ:c1prof} and \eqref{equ:c1limit}, we ask whether we can carry on these expansions and obtain lower order estimates. In particular in \eqref{equ:c1limit}, we wonder whether we can obtain the following terms of the expansion, up to bounded functions? In view of the self-similar transformation \eqref{def:simivars}, a necessary condition would be to carry on the expansion \eqref{equ:c1clas} up to the scale of $e^{-\frac{s}{p-1}} = (T-t)^\frac{1}{p-1}$. Unfortunately, any attempt to carry on the expansion \eqref{equ:c1clas} would give bunches of terms in the scale of powers of $\frac{1}{s} = \frac{1}{|\log (T-t)|}$ (with possibly $(\log s)^b$ corrections). This way, instead of reaching the scale of powers of the blow-up variable $(T-t)$, we are trapped in logarithmic scales of that variable, namely $\frac{1}{|\log(T-t)|^a}$. Logarithmic scales also arise in some singular perturbation problems such as low Reynolds number fluids and some vibrating membranes studies (see Ward \cite{Warbook96} and the references therein, see also Segur and Kruskal \cite{SKprl87} for a Klein-Gordon equation). Since the logarithmic scales go to zero slowly, infinite logarithmic series may be of only limited practical use in approximating the exact solution. Relevant approximations, i.e., approximations up to lower order terms $(T-t)^\beta$ for $\beta > 0$, lie beyond all logarithmic scales. In order to escape all logarithmic scales, a possible idea would be to abandon expansions around the explicit profile function \eqref{def:fk}, which happens to be only an approximate solution of equation \eqref{equ:w}, and to linearize around a non explicit profile function which is a solution of equation \eqref{equ:w}. This has been done by Fermanian and Zaag \cite{FZnon00} whose work shows that when linearizing around a fixed solution, say $\hat{u}$ a radially symmetric and decreasing solution to equation \eqref{equ:problem} which blows up in finite time $T$ at only $\hat{a} = 0$, they can reach the order $(T-t)^\beta$ for $\beta > 0$ through a modulation of the dilation of $\hat{u}$, provided that $N = 1$. In this paper, we aim at extending their result to the higher dimensional case. 

Let us explain the difficulty raised in \cite{FZnon00} for the case $N \geq 2$. It is convenient to introduce the following definitions:
\begin{defi}\label{def:setBaT} For all $(a, T) \in \RN \times \R$, we denote by $\mathbb{B}_{a,T}$ the set of all solutions to equation \eqref{equ:problem} which blow up in finite time $T$ at point $x = a$ (not necessary to be unique) and have the stable profile \eqref{equ:c1prof} (or \eqref{equ:c1clas} with $\ell = N$ or \eqref{equ:c1limit}). We denote by $\mathbb{B}'_{a,T}$ the subset of $\mathbb{B}_{a,T}$ where $a$ is the unique blow-up point and where no blow-up occurs at infinity (in the sense that $|u(x,t)| \leq C$ for all $|x| \geq c_0$ and $t \in [0,T)$ for some $C > 0$ and $c_0 > 0$).
\end{defi}
\begin{defi} We denote by $\mathbb{M}_N(\R)$ the set of all symmetric, real $(N \times N)$ matrices.
\end{defi}

\noindent Introducing the following dilation transformation for any $\lambda > 0$, 
\begin{equation}\label{equ:DilaInv}
\mathcal{D}_{\lambda}:\;  u \mapsto \mathcal{D}_{\lambda}u:\; (x,t) \mapsto \lambda^{\frac{2}{p-1}}u\big(\lambda x, T - \lambda^2(T-t)\big), 
\end{equation}
we see that $\mathcal{D}_\lambda$ is one-to-one from $\mathbb{B}_{a,T}$ to itself.

Let us consider $\hat{u}$, a radially symmetric and decreasing solution to equation \eqref{equ:problem}  in $\mathbb{B}'_{0,T}$ satisfying 
\begin{equation}\label{est:uhat}
\sup_{y \in \RN} \left|\T[\hat{u}](y,s) - \left[f\left(\frac{y}{\sqrt{s}}\right) + \frac{N\kappa}{2ps}\right] \right| \leq \frac{C}{\sqrt{s}},
\end{equation}
where $f$ is defined in \eqref{def:fk} (see Appendix \ref{ap:B} for the justification of the existence of such a solution). The solution $\hat{u}$ and $T$ will be considered as fixed in the following. Then, we have the following classification from \cite{FZnon00}: 

\medskip

\noindent \textit{If $u \in \mathbb{B}_{0,T}$, then, two cases arise:\\
- Case 1: There is a matrix $\mathcal{B} = \mathcal{B}(u, \hat{u}) \in \mathbb{M}_N(\R)$ ($\mathcal{B} \not \equiv 0$) such that
\begin{equation}\label{equ:case1FZ}
\T[u](y,s) - \T[\hat{u}](y,s) = \dfrac{1}{s^2}\left(\frac{1}{2}y^T\mathcal{B}y - tr(\mathcal{B})\right) + o\left(\frac{1}{s^2} \right) \quad \text{in}\;\; L^2_\rho.
\end{equation}
- Case 2: There is a constant $C > 0$, 
\begin{equation}\label{equ:case2FZ}
\|\T[u](s) - \T[\hat{u}](s)\|_{L^2_\rho} \leq \dfrac{Ce^{-s/2}}{s^3},
\end{equation}}
(see Appendix \ref{sec:apClas} for the justification of this result).

\medskip

When $N = 1$ ($\mathcal{B}(u, \hat{u}) \in \mathbb{R}$), the authors in \cite{FZnon00} noted the following property when $\lambda > 0$ and $u = \mathcal{D}_\lambda\hat{u}$ defined in \eqref{equ:DilaInv}:
\begin{equation}\label{equ:estFZuhat}
\T[\mathcal{D}_\lambda \hat{u}](y,s) - \T[\hat{u}](y,s) = \frac{\kappa \log \lambda}{ps^2}\left(\frac{1}{2}|y|^2 - 1 \right) + o\left(\frac{1}{s^2}\right) \quad \text{in}\; L^2_\rho,
\end{equation}
hence $\mathcal{B}(\mathcal{D}_\lambda \hat{u}, \hat{u}) = \frac{\kappa \log \lambda}{p}$ (note that \eqref{equ:estFZuhat} is true wherever $N \geq 2$ with $-1$ replaced by $-N$). This is due to the fact that 
$$
\T[\mathcal{D}_\lambda \hat{u}](s) = \T[\hat{u}](s + 2\log \lambda).
$$
Therefore, choosing $\lambda$ such that $\frac{\kappa \log \lambda}{p} = \mathcal{B}(u,\hat{u})$, that is $\lambda = e^{\frac{p}{\kappa} \mathcal{B}(u, \hat{u})}$, we see from \eqref{equ:case1FZ} and \eqref{equ:estFZuhat} that 
$$ \T[u](y,s) - \T[\mathcal{D}_{\lambda}\hat{u}](y,s) = o\left(\frac{1}{s^2}\right) \quad \text{in}\; L^2_\rho.$$
Hence, only \eqref{equ:case2FZ} holds and 
$$ \|\T[u](s) - \T[\mathcal{D}_{\lambda}\hat{u}](s) \|_{L^2_\rho} \leq \frac{C e^{-s/2}}{s^3}.$$
This implies by \cite{FZnon00} that when $p \geq 3$, 
\begin{equation}\label{equ:resultFZ00}
\left.\begin{array}{lll}
&|u(x,t) - \mathcal{D}_{\lambda}\hat{u}(x,t)| \leq C_0,&\quad \forall |x| \leq \epsilon_0, \quad \forall t\in \big[t_0 ,T \big), \\
&|u(x,t) - \mathcal{D}_{\lambda}\hat{u}(x,t)| \to 0 &\quad \text{as}\;\; (x,t) \to (0,T),
\end{array}\right\}
\end{equation}
where $\epsilon_0 > 0$ and $t_0 \in [0,T)$.

In view of \eqref{equ:resultFZ00}, it appears that the non explicit one-parameter family $\mathcal{D}_\lambda \hat{u}$ serves as a sharp blow-up final profile for any arbitrary $u \in \mathbb{B}_{0,T}$, accurate up to bounded functions. This is to be considered as a refinement of \eqref{equ:c1limit}, since $\mathcal{D}_\lambda \hat{u}$ encapsulates all singular terms in the expansion of $u(x,t)$ near the singularity $(0,T)$. However, there is a price to pay to reach such an accuracy, and the price lays in the fact that $\mathcal{D}_{\lambda}\hat{u}$ is not explicit, unlike $u^*(x)$ in \eqref{equ:c1limit}.
\bigskip

If $N \geq 2$, the matrix $\mathcal{B}(u, \hat{u})$ in \eqref{equ:case1FZ} has $\frac{N(N + 1)}{2}$ real parameters. Applying the dilation trick of \cite{FZnon00} allows to manage only one parameter. Therefore, $\frac{N(N+1)}{2} - 1$ parameters remain to be handled. This is the major reason preventing the authors in \cite{FZnon00} from having such a striking result in higher dimensions. Trying to apply other transformations which keep the equation and $\mathbb{B}_{0,T}$ invariant (rotation, symmetries of space coordinates, ...), we could not handle all the remaining $\frac{N(N+1)}{2} - 1$ parameters. Fortunately, we could overcome this obstacle and construct a $\frac{N(N+1)}{2}$ parameters family, which generalizes the $\mathcal{D}_{\lambda}\hat{u}$  family and serves as the accurate profile for solutions in $\mathbb{B}_{0,T}$. In the following statement, we construct that family:

\begin{theo}[\textbf{Construction of blow-up solutions for equation \eqref{equ:problem} in $\mathbb{B}'_{0,T}$ with a refined behavior}] \label{theo:1} For any $\mathcal{A} \in \mathbb{M}_N(\R)$, there exists $s_0(\mathcal{A}) > 0$ such that equation \eqref{equ:problem} has a unique solution $u_{\mathcal{A}}$ in $\mathbb{B}'_{0,T_{\mathcal{A}}}$ with $T_{\mathcal{A}} = e^{-s_0(\mathcal{A})}$ and the following holds  
\begin{equation}\label{equ:theo1refined}
\T[u_{\mathcal{A}}](y,s) - \T[\hat{u}](y,s) = \dfrac{1}{s^2}\left(\frac{1}{2}y^T\mathcal{A}y - tr(\mathcal{A})\right) + o\left(\frac{1}{s^2} \right)\;\;\text{in}\;\;L^2_\rho\;\; \text{as $s \to +\infty.$}
\end{equation}
\end{theo}

\begin{rema} From \eqref{equ:case1FZ}, we see that Theorem \ref{theo:1} remains true if we change $\hat{u}$ by any other $\tilde{u}$ in $\mathbb{B}'_{0,T}$.
\end{rema}
\begin{rema} The blow-up time $T_\mathcal{A}$ goes to zero when $\|\mathcal{A}\| \to +\infty$. 
\end{rema}

As mentioned earlier, Theorem \ref{theo:1} is a major step in extending \eqref{equ:resultFZ00} to the higher dimensional case. More precisely, we have the following result:
\begin{theo}[\textbf{A finite parameter family as a sharp profile for solutions of  \eqref{equ:problem} having the same profile \eqref{equ:c1prof}}] \label{theo:2} Consider $u \in \mathbb{B}_{0,T}$, then there exist a matrix $\mathcal{A} \in \mathbb{M}_N(\R)$, $\epsilon_0 > 0$ and $t_0 \in [0, T)$ such that \\
$(i) \qquad \qquad \qquad \|\T[u](s) - \T[\bar{u}_{\mathcal{A}}](s)\|_{L^2_\rho}  = \mathcal{O}\left(\dfrac{e^{-s/2}}{s^3}\right) \quad \text{as}\;\; s \to +\infty,$\\
where $\bar{u}_{\mathcal{A}}(x,t) = u_{\mathcal{A}}(x,t + T_\mathcal{A} - T)$ and $u_{\mathcal{A}} \in \mathbb{B}'_{0,T_{\mathcal{A}}}$ is the solution to equation \eqref{equ:problem} constructed in Theorem \ref{theo:1}. The convergence also holds in $L^\infty_{loc}$.\\

\noindent$(ii)$ For all $|x| \leq \epsilon_0$ and for all $t \in [t_0,T)$,
\begin{equation}\label{est:theo2p13}
|u(x,t) - \bar{u}_{\mathcal{A}}(x,t)\big| \leq C mM \left\{\frac{(T - t)^{\frac{1}{2} - \frac{1}{p-1}}}{|\log (T-t)|^\frac{3}{2}}, \frac{|x|^{1 - \frac{2}{p-1}}}{|\log |x||^{2 - \frac{1}{p-1}}} \right\},
\end{equation} 
where $mM = \min$ if $1 < p < 3$ and $mM = \max$ if $p \geq 3$.
\end{theo}
With this theorem, we see that if $p \geq 3$, then the difference $u - u_{\mathcal{A}}$ is bounded and goes to zero as $t \to T$, up to a good choice of $\mathcal{A}$ in $\mathbb{M}_N(\R)$, although both functions blow up. Therefore, Theorem \ref{theo:2} directly yields the following corollary:
\begin{coro}[\textbf{The sharp profile encapsulates all singular terms if $p \geq 3$}] \label{coro:3} Assume in addition to Theorem \ref{theo:2} that $p \geq 3$. Then 
\begin{equation}\label{est:corol7}
 \big |u(x,t) - u_{\mathcal{A}}(x,t + T_{\mathcal{A}} - T)\big| \leq C_0, \quad \forall |x| \leq \epsilon_0, \; \forall t \in [t_0,T),
\end{equation}
and 
$$\big |u(x,t) - u_{\mathcal{A}}(x,t + T_{\mathcal{A}} - T)\big| \to 0 \quad \text{as} \quad (x,t) \to (0,T).$$
\end{coro}

\begin{rema} If we denote by $\hat{\mathbb{B}}$ the set of solutions constructed in Theorem \ref{theo:1}, namely
\begin{equation}\label{def:BoT}
\hat{\mathbb{B}} = \hat{\mathbb{B}}(\hat{u}) \triangleq \{u_{\mathcal{A}} \in \mathbb{B}'_{0,T_{\mathcal{A}}} \; \text{constructed in Theorem \ref{theo:1} satisfying \eqref{equ:theo1refined}} \big \vert  \mathcal{A} \in \mathbb{M}_{N}(\R)\},
\end{equation}
and define from Corollary \ref{coro:3} the following equivalence relation $\thicksim$ on $\mathbb{B}_{0,T}$ for $p\geq3$: 
\begin{equation*}
\forall u,v \in \mathbb{B}_{0,T}, \quad u \thicksim v \Longleftrightarrow \exists \epsilon_0 > 0,\;\; (u - v) \in L^\infty\big(B(0,\epsilon_0) \times [T - \epsilon_0, T)\big),
\end{equation*}
then
$$u\diagup_{\thicksim}: \; \hat{\mathbb{B}} \thicksim \mathbb{M}_N(\R),$$
and 
$$\mathbb{B}_{0,T} \thicksim \mathbb{M}_N(\R) \times L^\infty_{x,t}.$$
This says that if we consider the blow-up asymptotic behavior given by \eqref{equ:c1clas} with $\ell = N$ or \eqref{equ:c1prof} or \eqref{equ:c1limit} as a first order expansion describing the behavior of $u(x,t)$ near the singular point $(0,T)$, then the following orders have $\frac{N(N+1)}{2}$ degrees of freedom which is the dimension of the set $\mathbb{M}_N(\R)$, up to bounded functions.
\end{rema}

\begin{rema} If $u$ blows up at time $T$ at some point $a \ne 0$ with the profile \eqref{equ:c1clas}, then $u(x - a,t) \in \mathbb{B}_{0,T}$. Thus, from Theorem \ref{theo:2} and Corollary \ref{coro:3}, we have a sharp profile for $u(x-a, t)$, hence for $u$. Note also that if $u \in \mathbb{B}'_{0,T}$, then estimates \eqref{est:theo2p13} and \eqref{est:corol7} hold for all $x \in \mathbb{R}^N$.

Not that Theorem \ref{theo:2} and Corollary \ref{coro:3} were already proved in one dimension by Fermanian and Zaag \cite{FZnon00}. Thus, the novelty of our contribution lays in the higher dimensional case.
\end{rema}

As in \cite{FZnon00}, we believe that our result is a forward step in the problem of the regularity of the blow-up set, which has been poorly studied in the literature and is challenging. In particular, Zaag in \cite{ZAAdm06} (see also \cite{ZAAaihp02} and \cite{ZAAcmp02}) used the ideas given in \cite{FZnon00} and proved that under a non-degeneracy condition, the blow-up set is a $\mathcal{C}^2$ manifold if it is continuous and its dimensional Hausdorff measure is equal to $N - 1$. He also derived the first description of the blow-up profile of solutions to \eqref{equ:problem} near a non-isolated blow-up point.\\

Let us now briefly give  the main ideas of the proof of Theorem \ref{theo:1}. The proof is based on techniques developed by Bricmont and Kupiainen in \cite{BKnon94}, Merle and Zaag in \cite{MZdm97} for the construction of a solution to equation \eqref{equ:problem} in $\mathbb{B}'_{0,T}$, that is, prescribing only the behavior \eqref{equ:c1prof}. Because we need in addition the estimate \eqref{equ:theo1refined} (note that this estimate is the crucial point in order to obtain Theorem \ref{theo:2}), we need new ideas. Instead of linearizing equation \eqref{equ:w} around $f\left(\frac{y}{\sqrt{s}}\right)$ defined in \eqref{def:fk} as in \cite{BKnon94} and \cite{MZdm97}, our major idea is to linearize equation \eqref{equ:w} around $\T[\hat{u}]$, where $\hat{u}$ is the given radially symmetric solution to equation \eqref{equ:problem} in $\mathbb{B}'_{0,T}$. Although this choice may seem less interesting, given that $\T[\hat{u}]$ is not explicit, unlike $f\left(\frac{y}{\sqrt{s}}\right)$, it is in fact much more advantageous, since linearizing around $\T[\hat{u}]$ generates no rest term, unlike with $f\left(\frac{y}{\sqrt{s}}\right)$. This way, we are able to reach the order $\frac{1}{s^2}$ in the expansion of solutions to equation \eqref{equ:w} (as expected in \eqref{equ:theo1refined}), unlike with $f\left(\frac{y}{\sqrt{s}}\right)$, where we are stuck in the $\frac{\log s}{s^2}$ order. Let us first review the method of \cite{BKnon94} and \cite{MZdm97} for the construction of a solution in $\mathbb{B}'_{0,T}$. In those papers, the proof is performed in the framework of similarity variables defined in \eqref{def:simivars}. In that setting, the problem reduces to the construction of a solution $w$ to \eqref{equ:w} such that 
$$v(y,s) = w(y,s) - f\left(\frac{y}{\sqrt{s}}\right) \to 0 \quad \text{as} \quad s \to +\infty.$$
Satisfying such a property is guaranteed by the condition that $v(s)$ belongs to some set $V_A(s) \subset L^\infty(\RN)$ which shrinks to 0 as $s \to +\infty$.  Since the linearization of equation \eqref{equ:w} around $f\left(  \frac{y}{\sqrt{s}}\right)$ gives $(N + 1)$ positive modes, $\frac{N(N+1)}{2}$ zero modes, then an infinite dimensional negative part, the method relies on two arguments:
\begin{itemize}
\item[-] The use of the bounding effect of the heat kernel to reduce the problem of the control of $v$ in $V_A$ to the control of its positive modes.
\item[-] The control of the $(N+1)$ positives modes thanks to a topological argument based on index theory.
\end{itemize}
Because the arguments of \cite{BKnon94} and \cite{MZdm97} allow the construction of solutions in $\mathbb{B}'_{0,T}$ for equation \eqref{equ:problem} without caring about estimate \eqref{equ:theo1refined}, therefore, we need some crucial modifications of the arguments of  \cite{MZdm97} in order to achieve additionally the estimate \eqref{equ:theo1refined} as well.  Although these modifications do not affect the general framework developed in \cite{MZdm97}, they lay in 3 crucial places:
\begin{itemize}
\item[i.] We do no longer linearize equation \eqref{equ:w} around the profile $f\left(\frac{y}{\sqrt{s}}\right)$ defined in \eqref{def:fk} as in \cite{BKnon94} and \cite{MZdm97}. we instead replace this explicit profile by an implicit one, say $\T[\hat{u}]$, where $\hat{u}$ is the radial solution to equation \eqref{equ:problem} in $\mathbb{B}_{0,T}'$. This way, we go beyond the $\frac{\log s}{s^2}$ order in the expansion of the solution and achieve the expected estimate \eqref{equ:theo1refined}. 
\item[ii.] The change of the definition of the shrinking set $V_A$ in a very delicate way, so that $v(s) \in V_A(s)$ implies $u \in \mathbb{B}'_{0,T}$ with estimate \eqref{equ:theo1refined} satisfied. With this change, we need to choose less explicit initial data $u_0$ so that the corresponding initial data of $v$, say $v(s_0)$, belongs to $V_A(s_0)$, unlike with \cite{MZdm97} where the initial data is given explicitly. See Section \ref{sec:defVAint}, particularly see Definition \ref{def:V_A} and Lemma \ref{lemm:decomdata}.
\item[iii.] In \cite{BKnon94} and \cite{MZdm97}, the $\frac{N(N+1)}{2}$ zero modes turned to be controllable like the negative modes, and this was made possible thanks to the effect of the linear potential term $\alpha v$ in \eqref{equ:v}. Here, because of the change of the definition of the shrinking set $V_A$ to satisfy \eqref{equ:theo1refined} as well, the $\frac{N(N+1)}{2}$ zero modes become in some sense "positive". This way, the topological argument concerns in all $N + 1 + \frac{N(N+1)}{2}$ terms.
\end{itemize}

We would like to mention that Masmoudi and Zaag \cite{MZjfa08} adapted the method of \cite{MZdm97} for the following Ginzburg-Landau equation:
\begin{equation}\label{equ:comGin}
\partial_t u = (1 + \imath \beta)\Delta u + (1 + \imath \delta)|u|^{p-1}u,
\end{equation}
where $p - \delta^2 - \beta\delta (p+1) > 0$ and $u: \mathbb{R}^N\times[0,T) \to \mathbb{C}$. Note that the case $\beta = 0$ and $\delta \in \mathbb{R}$ small has been  studied earlier by Zaag \cite{ZAAihn98}.  The same technique is successfully used by Nouaili and Zaag \cite{NZcpde15} for the following non-variational complex-valued semilinear heat equation:
$$\partial_t u = \Delta u + u^2,$$
where $u: \RN \times [0,T) \to \mathbb{C}$. In \cite{EZsema11}, Ebde and Zaag use these ideas to show the persistence of the profile \eqref{def:fk} under weak perturbations of equation \eqref{equ:problem} by lower order terms involving $u$ and $\nabla u$ (see also Nguyen and Zaag \cite{NZasnsp15} for the case of the strong perturbations). This kind of topological arguments has proved to be successful in various situations including hyperbolic and parabolic equations, in particular with energy-critical exponents. This was the case for the heat equation with exponential source by Bressan \cite{Breiumj90, Brejde92}, for the construction of multi-solitons for the semilinear wave equation in one space dimension by C\^ote and Zaag \cite{CZcpam13}, the wave maps by Rapha\"el and Rodnianski \cite{RRmihes12}, the Schr\"odinger maps by Merle, Rapha\"el and Rodnianski \cite{MRRmasp11}, the critical harmonic heat flow by Schweyer \cite{Schfa12} and the two-dimensional Keller-Segel equation by Rapha\"el and Schweyer \cite{RScpam13}.\\

As mentioned earlier, Theorem \ref{theo:1} is the major step in deriving Theorem \ref{theo:2} which actually extends \eqref{equ:resultFZ00} to the higher dimensional case. Let us briefly give the main steps of the proof of Theorem \ref{theo:2}. Consider $u$ in $\mathbb{B}_{0,T}$. Our goal is to choose a particular matrix $\mathcal{A} \in \mathbb{M}_N(\R)$ such that the difference $(T-t)^\frac{1}{p-1}\left[u(x,t) - u_{\mathcal{A}}(x,t + T_{\mathcal{A}} - T)\right]$, where $u_{\mathcal{A}} \in \mathbb{B}'_{0,T_{\mathcal{A}}}$ is the solution constructed in Theorem \ref{theo:1}, goes to zero in the scale of $(T-t)^\beta$ for some $\beta > 0$. In order to obtain this estimate, we follow the idea of \cite{FZnon00} in the one dimensional case and proceed in three steps:\\
- In the first step, we apply Theorem \ref{theo:2} with the matrix $\mathcal{A} = \mathcal{B}(u, \hat{u})$ given in the result of \cite{FZnon00} recalled in \eqref{equ:case1FZ}, hence, deriving the existence of $u_{\mathcal{A}}$ satisfying \eqref{equ:theo1refined}, we see that $\|\T[u](s) - \T[u_{\mathcal{A}}](s)\|_{L^2_\rho}$ goes to zero exponentially, and also in $L^\infty(|y| \leq R)$ for any $R > 0$ by  parabolic regularity.\\
- In the second step, we extend the estimate in compact sets to the larger sets $|y| \leq K\sqrt{s}$ by estimating the effect of the convective term $-\frac{y}{2}\cdot \nabla$ in the definition \eqref{def:opL} of $\mathcal{L}$ in $L^q_\rho$ spaces with $q > 1$. \\
- In the last step, we use a uniform ODE comparison result for equation \eqref{equ:problem} to estimate the difference $u(x,t) - u_{\mathcal{A}}(x,t + T_{\mathcal{A}} - T)$ in the outer region where $\epsilon_0 \geq |x| \geq K\sqrt{(T-t)|\log(T-t)|}$ for some $\epsilon_0 > 0$, and then get the conclusion.\\

\bigskip

\noindent We give the proof of Theorem \ref{theo:1} in Section \ref{sec:2}. The proof of Theorem \ref{theo:2} and Corollary \ref{coro:3} are given in Section \ref{sec:3}.

\section{Construction of blow-up solutions for \eqref{equ:problem} satisfying a prescribed behavior.}\label{sec:2}
This section is devoted to the proof of Theorem \ref{theo:1}. Consider $\hat{u} \in \mathbb{B}'_{0,T}$ the given radially symmetric
solution to equation \eqref{equ:problem} satisfying \eqref{est:uhat} and $\mathcal{A} \in \mathbb{M}_N(\R)$, we aim at constructing a solution $u_{\mathcal{A}}$ for equation \eqref{equ:problem} such that 
\begin{equation}\label{est:consMZ97}
\sup_{\xi \in \RN}\left|(T_{\mathcal{A}}-t)^\frac{1}{p-1}u_{\mathcal{A}}(\xi\sqrt{|\log (T_{\mathcal{A}}-t)|(T_{\mathcal{A}}-t)}, t) - f \left(\xi\right) \right| \leq \frac{C}{\sqrt{|\log (T_{\mathcal{A}}-t)|}}.
\end{equation}
where $f$ is defined in \eqref{def:fk}, and in the self-similar transformation \eqref{def:simivars}, it holds that 
\begin{equation}\label{est:crunew}
\T[u_{\mathcal{A}}](y,s) -\T[\hat{u}](y,s) = \frac{1}{s^2}\left(\frac{1}{2}y^T\mathcal{A}y - tr(\mathcal{A})\right) + o\left(\frac{1}{s^{2}}\right) \quad \text{in}\; L^2_\rho.
\end{equation}
If $\mathcal{A} = 0$, then, we simply take $u_{\mathcal{A}} = \hat{u}$ which already satisfies \eqref{est:consMZ97} as we explain in Appendix \ref{ap:B}. Therefore, we only consider here the case where 
\begin{equation}\label{equ:conMathA}
\mathcal{A} \ne 0.
\end{equation}

If $\hat{w} = \T[\hat{u}]$, in the similarity variables framework \eqref{def:simivars}, we reduce to finding $s_0 = s_0(\mathcal{A}) \in \mathbb{R}$ and $w_{\mathcal{A},0}(y)$ such that the solution  $w_{\mathcal{A}}(y,s)$ to equation \eqref{equ:w} with the initial datum $w_{\mathcal{A},0}$ exists for all $s \geq s_0$ and 
\begin{equation}\label{est:theo3_1}
\sup_{y \in \RN} \left|w_{\mathcal{A}}(y,s) - \hat{w}(y,s)\right| \to 0 \quad \text{as}\;\; s \to +\infty,
\end{equation}
with
\begin{equation}\label{est:theo3_2}
w_{\mathcal{A}}(y,s)- \hat{w}(y,s) = \frac{1}{s^2}\left(\frac{1}{2}y^T\mathcal{A}y - tr(\mathcal{A})\right) + o\left(\frac{1}{s^{2}}\right) \quad \text{in}\; L^2_\rho.
\end{equation}

Here, we follow the framework proposed in \cite{MZdm97} for the proof of weaker version of Theorem \ref{theo:1}, where only estimate \eqref{est:consMZ97} is needed, in particular estimate \eqref{est:theo3_2} is not considered and no solution $\hat{u}$ nor matrix $\mathcal{A}$ are needed. As in \cite{MZdm97}, the proof relies on the understanding of the dynamics of the self-similar version of equation \eqref{equ:w} around the function $\hat{w}$ with some refinement for the dynamics on the null mode to take care of \eqref{est:theo3_2}. This is indeed one of the major novelties in our work. More precisely, the proof is divided into 2 steps:
\begin{itemize}
\item[-] Thanks to a dynamical system formulation, we show that the control of the similarity variable version $w_{\mathcal{A}}(y,s)$ \eqref{def:simivars} around the sharper profile $\hat{w}$ given in \eqref{est:theo3_1} and \eqref{est:theo3_2} reduces to the control of the $N + 1$ positive modes and the $\frac{N(N+1)}{2}$ zero modes.
\item[-] Then, we solve the finite dimensional problem thanks to a topological argument based on index theory.
\end{itemize}

\noindent For the reader's convenience, we organize the proof in 4 subsections:\\
- In the first subsection, we formulate the constructive problem.\\
- In the second subsection, we give the definition of the shrinking set $V_A$ and the preparation of initial data for the problem.\\
- In the third subsection, we give all the arguments of the proof without the details, which are left for the following subsection.\\
- In the fourth subsection, we give the proof of an important proposition  which gives the reduction of the problem to a finite dimensional one.

\subsection{Formulation of the constructive problem.}
Consider $s_0 > 0$ to be fixed large enough later. Let us introduce the change of function
\begin{equation}\label{intr:v}
v_{\mathcal{A}}(y,s) = w_{\mathcal{A}}(y,s) - \hat{w}(y,s),
\end{equation}
where $\hat{w} = \T[\hat{u}]$ is the solution of \eqref{equ:w} which satisfies \eqref{est:uhat} and $\hat{u}$ is the considered radially decreasing solution of \eqref{equ:problem}. Then, from \eqref{equ:w}, $v_{\mathcal{A}}$ (or $v$ for simplicity) solves the following equation: for all $(y,s) \in \mathbb{R}^N \times [s_0, +\infty)$,
\begin{equation}\label{equ:v}
v_s = (\mathcal{L} + \gamma(y,s))v + B(v) = (\mathcal{L} + \alpha(y,s))v + B(v) + (\gamma(y,s) - \alpha(y,s))v, 
\end{equation}
where $\mathcal{L}$ is given in \eqref{def:opL} and
\begin{align}
\gamma(y,s) &= p \left(|\hat{w}(y,s)|^{p-1} - \kappa^{p-1}\right),\label{def:gamma}\\
B(v) &= |\hat{w} + v|^{p-1}(\hat{w} + v) - |\hat{w}|^{p-1}\hat{w} - p|\hat{w}|^{p-1}v,\label{def:B}\\
\alpha(y,s) &= p \left(|\varphi(y,s)|^{p-1} - \kappa^{p-1}\right), \quad \text{where}\;\; \varphi(y,s) = f\left(\frac{y}{\sqrt{s}}\right) + \frac{N\kappa}{2ps}.\label{def:alpha}
\end{align}
As mentioned earlier, we do linearize equation \eqref{equ:w} around $\hat{w}$, instead of the profile $f\left(\frac{y}{\sqrt{s}}\right)$. Doing $v$ generates no rest term in equation \eqref{equ:v} and this is one of the major ideas in this work. Looking at the second version of equation \eqref{equ:v}, the reader may ask why we use the function $\alpha(y,s)$ instead of $\gamma(y,s)$ as the potential. In fact, the use of the potential $\alpha$ is convenient for the two following reasons:
\begin{itemize}
\item[i.] We want to use the same dynamical system formulation given in \cite{BKnon94} and \cite{MZdm97}, and that analysis was already based on the understanding of the linear operator $\mL + \alpha$ and its related Duhamel formulation, together with some related  \textit{a priori} estimates that were already obtained (see Lemma \ref{lemm:BKespri} below for these estimates). 
\item[ii.] In view of \eqref{est:uhat}, we see from the definitions of $\alpha$ and $\gamma$ that they are almost the same in the sense that $\|\alpha(s) - \gamma(s)\|_{L^\infty} \to 0$ as $s \to +\infty$. Therefore, the term $(\gamma - \alpha)v$ in \eqref{equ:v} is easily controlled (see Lemma \ref{lemm:apestR} below). 
\end{itemize}

Satisfying \eqref{est:theo3_1} reduces to the construction of a function $v$ such that 
\begin{equation}\label{conv:vtozero}
\|v(s)\|_{L^{\infty}(\mathbb{R}^N)} \to 0 \quad \text{as} \; s\to+\infty.
\end{equation}
In fact, we will be more specific and require $v$ to satisfy some geometrical property, namely that $v$ belongs to some set $V_A \subset L^\infty(\RN)$ where $V_A(s)$ shrinks to $v \equiv 0$ as $s \to +\infty$. This set is very similar to that of \cite{MZdm97}, except for the control of the null modes, where we modify the definition of \cite{MZdm97} in a crucial way to handle the requirement given in \eqref{est:theo3_2}. In fact, our new definition covers the one of \cite{MZdm97}. Again, we insist on the fact that this is our second main contribution and novelty in this work, with respect to \cite{MZdm97} (see Definition \ref{def:V_A} for more clarity, especially condition \eqref{equ:conVAnullmode} below).\\

Our analysis uses the Duhamel formulation of equation \eqref{equ:v}: for each $s \geq \sigma \geq s_0$, we have
\begin{equation}\label{for:vint}
v(s) = \mathcal{K}(s,\sigma)v(\sigma) + \int_\sigma^s \mathcal{K}(s, \tau)\left[B(v(\tau)) + (\gamma(\tau) - \alpha(\tau))v(\tau)\right] d\tau,
\end{equation}
where $\mathcal{K}$ is the fundamental solution of the linear operator $\mathcal{L} + \alpha$ defined for each $\sigma > 0$ and $s \geq \sigma$ by
\begin{equation} \label{def:kernel}
\partial_s\mathcal{K}(s,\sigma) = (\mathcal{L} + \alpha)\mathcal{K}(s,\sigma), \quad \mathcal{K}(\sigma, \sigma) = Identity.
\end{equation}

\noindent The linear operator $\mathcal{L}$ is self-adjoint in $L^2_\rho(\mathbb{R}^N)$, where $L^2_\rho$ is the weighted $L^2$ space associated with the weight $\rho$ defined by 
\begin{equation}\label{def:rho}
\rho(y) = \prod_{i = 1}^N\rho_1(y_i) \quad \text{with} \quad \rho_1(\xi) =\frac{1}{\sqrt{4\pi}}e^{-\frac{|\xi|^2}{4}},
\end{equation}
and 
$$spec(\mathcal{L}) = \{1 - \frac{n}{2},\; n \in \mathbb{N}\}.$$
For $\beta = (\beta_1, \cdots, \beta_N) \in \mathbb{N}^N$, the eigenfunctions corresponding to $1 - \frac{|\beta|}{2}$ ($|\beta| = \beta_1 + \cdots + \beta_N$) are
\begin{equation}\label{def:Phim}
\phi_\beta(y) = \phi_{\beta_1}(y_1)\cdots\phi_{\beta_N}(y_N), 
\end{equation}
where 
\begin{equation}\label{def:phik}
\phi_k(\xi) = \sum_{i = 0}^{\left[\frac{k}{2} \right]} \frac{k!}{i!(k - 2i)!}(-1)^i\xi^{k - 2i}, \quad k \in \mathbb{N},
\end{equation}
satisfy
\begin{equation}\label{equ:orthre}
\int_{\mathbb{R}} \phi_k(\xi) \phi_n(\xi) \rho_1(\xi) d\xi = 2^k k!\delta_{k,n}.
\end{equation}

\noindent Note from Lemma \ref{lemm:apEstalp} that the potential $\alpha(y,s)$ has two fundamental properties: \textit{
\begin{itemize}
\item[i)] $\alpha(\cdot, s) \to 0$ in $L^2_{\rho}$ as $s \to +\infty$. In particular, the effect of $\alpha$ on the bounded sets or in the "blow-up" region ($|y| \leq K\sqrt{s}$) is regarded as a perturbation of the effect of $\mathcal{L}$.
\item[ii)] outside of the "blow-up" region, we have the following property: for all $\epsilon > 0$, there exist $C_\epsilon > 0$ and $s_\epsilon$ such that 
\begin{equation}\label{equ:asymalp}
\sup_{s \geq s_\epsilon, |y| \geq C_\epsilon \sqrt{s}} \left|\alpha(y,s) - \left(-\frac{p}{p-1}\right)\right| \leq \epsilon.
\end{equation}
\end{itemize}}

\noindent This means that $\mathcal{L} + \alpha$ behaves like $\mathcal{L} - \frac{p}{p-1}$ in the region $|y| \geq K\sqrt{s}$. Because $1$ is the largest eigenvalue of $\mathcal{L}$, the operator $\mathcal{L} - \frac{p}{p-1}$ has a purely negative spectrum. Therefore, the control of $v(y,s)$ in $L^\infty$ outside of the "blow-up" region will be done without difficulties. \\

Since the behavior of $\alpha$ inside and outside of the "blow-up" region is different, let us decompose $v$ as follows: Let $\chi_0 \in \mathcal{C}_0^\infty([0,+\infty))$ with $\text{supp}(\chi_0) \subset [0,2]$ and $\chi_0 \equiv 1$ on $[0,1]$. We define 
\begin{equation}\label{def:chi}
\chi(y,s)= \chi_0\left(\frac{|y|}{K\sqrt{s}}\right),
\end{equation}
where $K > 0$ is to be fixed large enough, and write 
\begin{equation}\label{de:vbve}
v(y,s) = v_b(y,s) + v_e(y,s),
\end{equation}
where 
$$v_b(y,s) = \chi(y,s) v(y,s)\quad \text{and} \quad v_e(y,s) = (1 - \chi(y,s))v(y,s).$$
Note that $supp (v_b(s)) \subset \mathbf{B}(0,2K\sqrt{s})$ and $supp (v_e(s)) \subset \mathbb{R}^N\setminus\mathbf{B}(0,K\sqrt{s})$.

In order to control $v_b$, we expand it with respect to the spectrum of $\mL$ in $L^2_\rho$ since the eigenfunctions of $\mL$ span the whole space $L^2_\rho(\mathbb{R}^N)$. More precisely, we write $v$ as follows:
\begin{align}\label{def:decomv}
v(y,s) = v_{0}(s) + v_{1}(s)\cdot y + \frac{1}{2}y^Tv_{2}\,y - tr(v_{2}) + v_{-}(y,s) + v_e(y,s),
\end{align}
where 
$v_{0}(s) = P_0(v_b)(y,s)$, $v_1(s)\cdot y = P_1(v_b)(y,s)$, $v_{-}(y,s) = P_-(v_b)(y,s) = \sum_{m \geq 3} P_m(v_b)(y,s)$, and $P_m$ is the projector on the eigenspace corresponding to the eigenvalue $1 - \frac{m}{2}$ defined by
\begin{equation}\label{def:Projector}
P_m(v_b)(y,s) = \sum_{\beta \in \mathbb{N}^N, |\beta| = m}\frac{\phi_\beta(y)}{\|\phi_\beta\|^2_{L^2_\rho}}\int_{\mathbb{R}^N} \phi_\beta(y)v_b(y,s)\rho(y)dy,
\end{equation}
where $\phi_\beta$ is defined in \eqref{def:Phim}, and $v_{2}(s) \in \mathbb{M}_N(\R)$ defined by
\begin{equation}\label{def:vb2}
v_{2}(s) = \int_{\mathbb{R}^N}v_b(y,s)\mathcal{M}(y)\rho(y)dy,
\end{equation}
where 
\begin{equation}\label{def:M2}
\mathcal{M}(y) = \left\{\frac{1}{4}y_iy_j - \frac{1}{2}\delta_{ij} \right\}_{1\leq i,j \leq N}.
\end{equation}
The reader should keep in mind that $v_m, m = 0, 1, 2$ and $v_-$ are coordinates of $v_b$ and not those of $v$.

\subsection{Definition of a shrinking set $V_A(s)$ and preparation of initial data.}\label{sec:defVAint}
Our two requirements \eqref{est:theo3_1} and \eqref{est:theo3_2} follow directly if we construct a solution $v(s)$ of equation \eqref{equ:v} such that $v(s)$ belongs to a set $V_A(s)$ for some $s_0 \geq 1$, where $V_A(s)$ is defined in the following:
\begin{defi}[\textbf{A shrinking set to zero}] \label{def:V_A} Let $\eta \in \left(0,\frac{1}{2}\right)$, for each $A > 0$, for each $s > 0$, we define $V_A(s)$ as being the set of all functions $g$ in $L^\infty(\mathbb{R}^N)$ such that
\begin{align}
|g_0(s)| \leq \frac{A}{s^{2+\eta}}, \quad |g_{1,i}(s)| &\leq \frac{A}{s^{2+\eta}}, \quad \forall i \in \{1,\cdots, N\},\nonumber\\
\left|g_{2,ij}(s) - \frac{a_{ij}}{s^2}\right| &\leq \frac{A^2}{s^{2 + \eta}}, \quad \forall i,j \in \{1, \cdots, N\}, \label{equ:conVAnullmode}\\
\forall y \in \mathbb{R}^N, \; |g_-(y,s)| &\leq \frac{A}{s^{2+\eta}}(1 + |y|^3),\nonumber\\
\|g_e(s)\|_{L^\infty} &\leq \frac{A^2}{s^{1/2 + \eta}},\nonumber
\end{align}
where $g_0, g_{1,i}, g_{2,ij}, g_-$ and $g_e$ are defined as in \eqref{def:decomv}, $a_{ij}$'s are the coefficient of the given matrix $\mathcal{A}$.\\
We also define $\hat{V}_A(s) \subset \mathbb{R} \times \mathbb{R}^N \times \mathbb{M}_N(\mathbb{R})$ as follows: 
$$\hat{V}_A(s) = \left[-\frac{A}{s^{2+\eta}}, \frac{A}{s^{2+\eta}}\right]\times\left[-\frac{A}{s^{2+\eta}}, \frac{A}{s^{2+\eta}}\right]^{N} \times \left\{\mathbb{M}_N\left(\left[-\frac{A^2}{s^{2 + \eta}}, \frac{A^2}{s^{2 + \eta}} \right]\right) + \frac{\mathcal{A}}{s^2} \right\}.$$
\end{defi}
\begin{rema}\label{rema:coverV_A} 
In \cite{MZdm97}, the shrinking set was very similar in the sense that one has to take $\eta = 0$ above and to replace the condition  \eqref{equ:conVAnullmode} by 
\begin{equation}\label{equ:conVAMZg2}
\forall i,j \in \{1, \dots, N\}, \;\; |g_{2,ij}(s)| \leq \frac{A^2 \log s}{s^2}.
\end{equation}
This way, Definition \ref{def:V_A} and especially \eqref{equ:conVAnullmode} appear as the originality in our strategy. Let us note that our shrinking set $V_A(s)$ is included in \cite{MZdm97}, provided that $s$ is large enough (with respect to the matrix $\mathcal{A}$).
\end{rema}

In order to see that the requirements \eqref{est:theo3_1} and \eqref{est:theo3_2} are fulfilled when $v(s) \in V_A(s)$ for all $s \geq s_0$, we write from \eqref{def:decomv},
\begin{align*}
v(y,s) &= \left\{v_{0}(s) + v_{1}(s)\cdot y + \frac{1}{2}y^Tv_{2}\,y - tr(v_{2}) + v_{-}(y,s)\right\}\cdot\textbf{1}_{\{|y| \leq 2K\sqrt{s}\}}  + v_e(y,s),
\end{align*}
which gives by Definition \ref{def:V_A}
\begin{equation}\label{equ:Con1for20}
\sup_{y \in \mathbb{R}^N}|v(y,s)| \leq \frac{C(A)}{s^{1/2 + \eta}},
\end{equation}
hence \eqref{conv:vtozero} and \eqref{est:theo3_1}.

As for \eqref{est:theo3_2}, we see from \eqref{equ:conVAnullmode} that 
\begin{equation}\label{equ:Con2for21}
w_{\mathcal{A},2}(s) - \hat{w}_2(s) = v_2(s) = \frac{1}{s^2}\mathcal{A} + \mathcal{O}\left(\frac{1}{s^{2 + \eta}}\right)
\end{equation}
on the one hand. On the other hand, introducing $u_{\mathcal{A}}$ the solution to equation \eqref{equ:problem} which blows up at time $T_{\mathcal{A}} = e^{-s_0}$ such that $\T[u_{\mathcal{A}}] = w_\mathcal{A} = \hat{w} + v$. From the classification result of \cite{FZnon00} given in page \pageref{equ:case1FZ}, we see that  case 2 does not hold, otherwise we would have by projection $w_{\mathcal{A},2}(s)  - \hat{w}_2(s) = \mathcal{O}\left(\frac{e^{-s/2}}{s^3}\right)$. Hence, $\mathcal{A} = 0$ from \eqref{equ:Con2for21}, which is a contradiction from \eqref{equ:conMathA}. Therefore, only case 1 holds, and we have
\begin{equation}\label{equ:tmpProth1}
w_{\mathcal{A}}(y,s) - \hat{w}(y,s) = \frac{1}{s^2}\left(\frac{1}{2}y^T\mathcal{B}y - tr(\mathcal{B}) \right) + o\left(\frac{1}{s^{2}}\right)
\end{equation}
for some $\mathcal{B} = \mathcal{B}(u_\mathcal{A},\hat{u})$. Therefore, projecting on the null modes, we get 
$$w_{\mathcal{A},2}(s) - \hat{w}_2(s) = \frac{1}{s^2}\mathcal{B} + o\left(\frac{1}{s^{2}}\right).$$
From \eqref{equ:Con2for21}, it follows that $\mathcal{A} = \mathcal{B}(u_\mathcal{A}, \hat{u})$. Thus, \eqref{est:theo3_2} follows from \eqref{equ:tmpProth1}.\\

\bigskip

Our goal then becomes to construct a solution $v(s)$ of equation \eqref{equ:v} such that 
$$v(s) \in V_A(s), \quad \text{for all }\;\; s \geq s_0,$$
for some $s_0$. Let us first give the general form we take for initial data to fulfill this requirement. Initial data (at time $s_0$) for the equation \eqref{equ:v} will depend on a finite number of real parameters $d_0$, $d_{1,i}$ and $d_{2,ij}$ with $1 \leq i,j \leq N$ as given in the following lemma:
\begin{lemm}[\textbf{Decomposition of initial data on the different components)}] \label{lemm:decomdata} For each $A > 1$, there exists $\delta_1(A) > 0$ such that for all $s_0 \geq \delta_1(A)$: If we consider the following function as initial data for equation \eqref{equ:v}:
\begin{equation}\label{def:intv0}
v_{d_0,d_1,d_2}(y,s_0) = \left(\frac{A}{s_0^{2+\eta}}(d_0 + d_1\cdot y) + \frac{1}{2}y^T\hat{d}_2\,y -2tr(\hat{d}_2)\right)\chi(2y,s_0),
\end{equation}
where 
$$\hat{d}_{2,ij} = \frac{a_{ij}}{s_0^2} + \frac{A^2d_{2,ij}}{s_0^{2 + \eta}},$$
and $\chi$ is defined in \eqref{def:chi}, then, the following holds: 
\begin{itemize}
\item[(i)] If $|d_0|+ |d_{1}|+ |d_{2}| \leq 2$, then, the components of $v_{d_0,d_1,d_2}(s_0)$ (or $v(s_0)$ for short) satisfy:
\begin{align*}
&\left|v_0(s_0) - \frac{Ad_0}{s_0^{2+\eta}}\right| \leq Ce^{-s_0}, \quad \left|v_{1,i}(s_0) - \frac{Ad_{1,i}}{s_0^{2+\eta}} \right| \leq Ce^{-s_0}, \quad \forall i\in \{1,\cdots,N\},\\
&\left|v_{2,ij}(s_0) - \frac{a_{ij}}{s_0^2} - \frac{A^2d_{2,ij}}{s_0^{2 + \eta}}\right| \leq C e^{-s_0}, \quad \forall i,j\in \{1,\cdots,N\},\\
&\left|v_{-}(y,s_0)\right| \leq C\left(\frac{|d_0| + |d_1| + |d_2| + \|\mathcal{A}\|}{s_0^{5/2}} \right)(1 + |y|^3),\\
&v_e(y,s_0) \equiv 0.
\end{align*}
\item[(ii)] If $(d_0, d_1, d_2)$ is chosen such that $(v_0, v_1, v_2)(s_0) \in \hat{V}_A(s_0)$, then 
\begin{align*}
&|d_0| + |d_1| + |d_2| \leq 2, \\
&\left\|\frac{v_-(s_0)}{1 + |y|^3} \right\|_{L^\infty} \leq \frac{C}{s_0^{5/2}}, \quad \|v_e(s_0)\|_{L^\infty} = 0,
\end{align*}
and $v(s_0)\in V_A(s_0)$ with "strict inequalities", except for $(v_0, v_1, v_2)(s_0)$, in the sense that
\begin{align*}
|v_0(s_0)| \leq \frac{A}{s_0^{2+\eta}}, \quad |v_{1,i}(s_0)| &\leq \frac{A}{s_0^{2+\eta}}, \quad \forall i \in \{1,\cdots, N\},\\
\left|v_{2,ij}(s_0) - \frac{a_{ij}}{s_0^2}\right| &\leq \frac{A^2}{s_0^{2 + \eta}}, \quad \forall i,j \in \{1, \cdots, N\},\\
\forall y \in \mathbb{R}^N, \; |v_-(y,s_0)| & < \frac{A}{s_0^{2+\eta}}(1 + |y|^3),\\
\|v_e(s_0)\|_{L^\infty} & < \frac{A^2}{s_0^{1/2 + \eta}}.
\end{align*}
\item[(iii)] There exists a subset $\mathcal{D}_{s_0} \subset \R \times \RN \times \mathbb{M}_N(\R)$ such that the mapping 
$$(d_0, d_1, d_2) \mapsto (v_0, v_1, v_2)(s_0)$$
is linear  and one to one  from $\mathcal{D}_{s_0}$ on to $\hat{V}_A(s_0)$ and maps $\partial \mathcal{D}_{s_0}$ into $\partial \hat{V}_A(s_0)$. Moreover, it is of degree one on the boundary and the following equivalence holds:
$$v(s_0) \in V_A(s_0) \quad \text{if and only if}  \quad (d_0,d_1,d_2) \in \mathcal{D}_{s_0}.$$
\end{itemize}
\end{lemm}
\begin{proof} For parts $(i)$ and $(ii)$, the proof is purely technical and follows from the definition \eqref{def:decomv}. For details in a similar case, see Nouaili and Zaag \cite{NZcpde15}. Part $(iii)$ follows from the first three estimates in part $(i)$, part $(ii)$ and Definition \ref{def:V_A} of $V_A$. This ends the proof of Lemma \ref{lemm:decomdata}.
\end{proof}

\subsection{Reduction to a finite dimensional problem and conclusion of Theorem \ref{theo:1}.}

Let us state the following central proposition which implies Theorem \ref{theo:1}:
\begin{prop}[\textbf{Sufficient condition for Theorem \ref{theo:1}}] \label{prop:equivTheo1} There exist $A > 1$ and $S_0 > 0$ such that for all $s_0 \geq S_0$, there exists $(d_0,d_1,d_2) \in \mathcal{D}_{s_0}$ such that the equation \eqref{equ:v} with initial data at $s = s_0$ given by $v_{d_0,d_1,d_2}(y,s_0)$ \eqref{def:intv0}, has a unique solution $v_{d_0,d_1,d_2}(s)$ defined for $s \geq s_0$ and satisfying 
$$v(s) \in V_A(s), \quad \forall s \geq s_0.$$
\end{prop}

Let us first give the proof of Proposition \ref{prop:equivTheo1}, then the proof of Theorem \ref{theo:1} will be given later. The proof of Proposition \ref{prop:equivTheo1} follows from the general ideas developed in \cite{MZdm97}. It is divided in two parts:\\
- In the first part, we reduce the problem of the control $v(s)$ in $V_A(s)$ to the control of $(v_0, v_1, v_2)(s)$, which are the components of $v$ corresponding to the positive and null modes given in expansion \eqref{def:decomv}. That is, we reduce an infinite dimensional problem to a finite dimensional one.\\
- In the second part, we solve the finite dimensional problem, using dynamics of $(v_0, v_1, v_2)(s)$ and a topological argument based on the variation of the finite dimensional parameters $(d_0, d_1, d_2)$ appearing in the expression \eqref{def:intv0} of initial data $v_{d_0,d_1,d_2}(y,s_0)$.\\

\subsubsection*{\textbf{Part I: Reduction to a finite dimensional problem.}}
In this step, we first show through a priori estimates that the control of $v(s)$ in $V_A(s)$ reduces to the control of $(v_0,v_1, v_2)(s)$ in $\hat{V}_A(s)$. As presented in \cite{MZdm97} (see also \cite{ZAAihn98}, \cite{MZjfa08}, \cite{NZcpde15}, \cite{NZasnsp15}), this step makes the heart of our contribution. We mainly claim the following:

\begin{prop}[\textbf{Control of $v(s)$ by $(v_0,v_1,v_2)(s)$ in $\hat{V}_A(s)$}] \label{prop:redu} There exist $A_3 > 0$ such that for each $A \geq A_3$, there exists $\delta_3(A) > 0$ such that for each $s_0 \geq \delta_3(A)$, we have the following properties:\\
- if $(d_0, d_1, d_2)$ is chosen so that $(v_0, v_1, v_2)(s_0) \in \hat{V}_A(s_0)$, and\\
- if for all $s \in [s_0,s_1]$, $v(s) \in V_A(s)$ and $v(s_1) \in \partial V_A(s_1)$ for some $s_1 \geq s_0$, then 
\begin{itemize}
\item[(i)] \textbf{(Reduction to a finite dimensional problem)} $\;\; (v_0, v_1,v_2)(s_1) \in \partial \hat{V}_A(s_1)$.
\item[(ii)] \textbf{(Transversality)} There exists $\mu_0 > 0$ such that for all $\mu \in (0, \mu_0)$, 
$$(v_0, v_1, v_2)(s_1 + \mu) \not\in \hat{V}_A(s_1 + \mu)\;\text{(hence, $v(s_1 + \mu) \not\in V_A(s_1 + \mu)$).}$$
\end{itemize}
\end{prop}
\begin{proof} Since we would like to keep the proof of Proposition \ref{prop:equivTheo1} short, we leave the proof of Proposition \ref{prop:redu} to the next subsection. 
\end{proof}
\subsubsection*{\textbf{Part II: Topological argument for the finite dimensional problem.}}
In the following proposition, we study the Cauchy problem for equation \eqref{equ:v}.
\begin{prop}[\textbf{Local in time solution of equation \eqref{equ:v}}] \label{prop:localex} For all $A > 1$, there exists $\delta_5(A)$ such that for all $s_0 \geq \delta_5(A)$, the following holds: For all $(d_0,d_1,d_2) \in \mathcal{D}_{s_0}$, there exists $s_{max}(d_0,d_1,d_2) > s_0$ such that equation \eqref{equ:v} with initial data $v_{d_0,d_1,d_2}(s_0)$ given in \eqref{def:intv0} has a unique solution satisfying $v(s) \in V_{A+1}(s)$ for all $s \in [s_0,s_{max})$.
\end{prop}
\begin{proof} Using the definition  \eqref{intr:v} of $v$ and the similarity variables transformation \eqref{def:simivars}, we see that the Cauchy problem of \eqref{equ:v} is equivalent to the Cauchy problem of equation \eqref{equ:problem}. Note that the initial data for \eqref{equ:problem} is derived from the initial data for \eqref{equ:v} at $s=s_0$ given in \eqref{def:intv0}, and it belongs to $L^\infty(\R)$, which insures the local existence of $u$ in $L^\infty(\R)$ (see the introduction). From part $iii)$ of Lemma \ref{lemm:decomdata}, we have $v_{d_0,d_1,d_2}(s_0) \in V_A(s_0) \subseteq V_{A + 1}(s_0)$. Then there exists $s_{max}$ such that for all $s \in [s_0,s_{max})$, we have $v(s) \in V_{A+1}(s)$. This concludes the proof of Proposition \ref{prop:localex}.
\end{proof}
Let us now derive the conclusion of Proposition \ref{prop:equivTheo1}, assuming Proposition \ref{prop:redu}. Although the derivation of the conclusion is the same as in \cite{MZdm97}, we would like to give details of the proof for the reader's convenience.
\begin{proof}[\textbf{Proof of Proposition \ref{prop:equivTheo1},  assuming Proposition \ref{prop:redu}}] Let us take $A \geq A_1$ and $s_0 \geq \delta_3$, where $A_1$ and $\delta_3$ are given in Proposition \ref{prop:redu}. We will find a parameter $(d_0,d_1,d_2)$ in the set $\mathcal{D}_{s_0}$ defined in Lemma \ref{lemm:decomdata} such that 
$$v_{d_0, d_1, d_2}(s) \in V_A(s), \quad \forall s \in [s_0, +\infty),$$
where $v_{d_0,d_1,d_2}$ is the solution to equation \eqref{equ:v} with initial data given in \eqref{def:intv0}.

We proceed by contradiction. From $(iii)$ of Lemma \ref{lemm:decomdata}, this means that for all $(d_0,d_1,d_2) \in \mathcal{D}_{s_0}$, there exists $s_*(d_0,d_1,d_2) \geq s_0$ such that $v_{d_0,d_1, d_2}(s) \in V_A(s)$ for all $s \in [s_0,s_*]$ and $v_{d_0,d_1,d_2}(s_*) \in \partial V_A(s_*)$. Applying item $(i)$ in Proposition \ref{prop:redu}, we see that $v_{d_0,d_1,d_2}(s_*)$ can leave $V_A(s_*)$ only by its first three components,  that is
$$(v_0,v_1,v_2)(s_*) \in \partial \hat{V}_A(s_*).$$
Therefore, we can define the following function: 
\begin{align*}
\Phi \;: \mathcal{D}_{s_0} &\mapsto \partial ([-1,1] \times [-1,1]^N \times \mathbb{M}_N([-1,1]))\\
(d_0, d_1, d_2) &\to \left(\frac{s_*^{2+\eta}}{A}v_0(s_*), \frac{s_*^{2+\eta}}{A}v_1(s_*), \frac{s_*^{2+\eta}}{A^2} \left(v_2(s_*) + \frac{\mathcal{A}}{s_*^2} \right) \right).
\end{align*}
Since $v(y,s)$ is continuous in $(d_0,d_1,d_2, s)$ (see Lemma \ref{lemm:decomdata} and Proposition \ref{prop:localex}), it follows that $(v_0,v_1,v_2)(s)$ is continuous with respect to $(d_0,d_1,d_2,s)$ too. Then, using the transversality property of $(v_0,v_1,v_2)$ on $\partial \hat{V}_A$ (part $(ii)$ of Proposition \ref{prop:redu}), we see that $s_*(d_0,d_1,d_2)$ is continuous. Therefore, $\Phi$ is continuous.

If we manage to prove that $\Phi$ is of degree one on the boundary, then we have a contradiction from the degree theory. Let us prove that. From item $(iii)$ in Lemma \ref{lemm:decomdata}, we see that if $(d_0,d_1,d_2)$ is on the boundary of $\mathcal{D}_{s_0}$, then 
$$v(s_0) \in V_A(s_0) \quad \text{and}\quad (v_0,v_1,v_2)(s_0) \in \partial \hat{V}_A(s_0).$$
Using $(ii)$ of Proposition \ref{prop:redu}, we see that $v(s)$ must leave $V_A(s)$ at $s = s_0$, hence $s_*(d_0,d_1,d_2) = s_0$ and $\Phi(d_0, d_1, d_2) = \left(\frac{s_0^{2 + \eta}}{A}v_0(s_0), \frac{s_0^{2 + \eta}}{A}v_1(s_0), \frac{s_0^{2 + \eta}}{A^2}(v_2(s_0) + \frac{\mathcal{A}}{s_0^2})\right)$. Using again $(iii)$ of Lemma \ref{lemm:decomdata}, we see that the restriction of $\Phi$ to the boundary is of degree $1$. This gives us a contradiction (by the index theory). Thus, there exists $(d_0,d_1,d_2) \in \mathcal{D}_{s_0}$ such that for all $s \geq s_0$, $v_{d_0,d_1,d_2}(s) \in V_A(s)$, which is the conclusion of Proposition \ref{prop:equivTheo1}.
\end{proof}

Let us now derive Theorem \ref{theo:1} from Proposition \ref{prop:equivTheo1}, assuming Proposition \ref{prop:redu}.
\begin{proof}[\textbf{Proof of Theorem \ref{theo:1} from Proposition \ref{prop:equivTheo1}, assuming Proposition \ref{prop:redu}}] Applying Proposition \ref{prop:equivTheo1} with $s_0 = S_0$, we derive the existence of $v_{\mathcal{A}}(s) \in V_A(s)$ for all $s \geq S_0$. Let us introduce $w_{\mathcal{A}}$ the solution of \eqref{equ:w} such that
$$w_{\mathcal{A}}(y,s) = \hat{w}(y,s) + v_{\mathcal{A}}(y,s),$$
then $u_{\mathcal{A}}$ the solution of equation \eqref{equ:problem} such that 
$$\T[u_\mathcal{A}] = w_{\mathcal{A}}.$$
From the arguments given around \eqref{equ:Con1for20} and \eqref{equ:Con2for21}, we have proved that $w_{\mathcal{A}}$ satisfies \eqref{est:theo3_1} and \eqref{est:theo3_2}, hence $u_{\mathcal{A}}$ satisfies \eqref{est:consMZ97} and \eqref{est:crunew}. It remains to show that $u_{\mathcal{A}}$ blows up only at the origin. To this end, let us remark from \eqref{est:consMZ97} that 
$$(T_{\mathcal{A}}-t)^\frac{1}{p-1}u_{\mathcal{A}}(0,t) \sim f(0) = \kappa,$$
and 
$$\forall x_0 \ne 0, \quad (T_{\mathcal{A}}-t)^\frac{1}{p-1}u_{\mathcal{A}}(x_0,t) \to 0, \quad \quad \text{as} \; t \to T_{\mathcal{A}}.$$
From the classification result of Giga and Kohn \cite{GKcpam89}, this implies that $u_{\mathcal{A}}$ blows up only at the origin. Hence, $u_{\mathcal{A}} \in \mathbb{B}'_{0,T_{\mathcal{A}}}$ with \eqref{equ:theo1refined} satisfied. This concludes the proof of Theorem \ref{theo:1}, assuming Proposition \ref{prop:redu} holds.
\end{proof}

\subsection{Proof of Proposition \ref{prop:redu}.}
We give in this subsection the proof of Proposition \ref{prop:redu} in order to complete the proof of Theorem \ref{theo:1}. The proof follows the ideas of \cite{MZdm97} and we proceed in three steps:
\begin{itemize}
\item[-] Step 1: we give a priori estimates  on $v(s)$ in $V_A(s)$: assume that for given $A > 0$ large, $\lambda > 0$ and an initial time $s_0 \geq \sigma_2(A,\lambda) \geq 1$, we have $v(s) \in V_A(s)$ for each $s \in [\tau, \tau + \lambda]$ where $\tau \geq s_0$, then using the integral form \eqref{for:vint} of $v(s)$, we derive new bounds on $v_-(s)$ and $v_e(s)$ for $s \in [\tau, \tau + \lambda]$.
\item[-] Step 2: we show that these new bounds are better than those defining $V_A(s)$. It then remains to control $(v_0, v_1, v_2)(s)$. This means that the problem is reduced to the control  of a finite dimensional function $(v_0, v_1, v_2)(s)$ and then we get the conclusion $(i)$ of Proposition \ref{prop:redu}.
\item[-] Step 3: we derive from \eqref{equ:v} differential equations  satisfied by $(v_0, v_1, v_2)(s)$ to show its transversality on $\partial \hat{V}_A(s)$, which yields the conclusion $(ii)$ of Proposition \ref{prop:redu}.
\end{itemize}

\subsubsection*{\textbf{Step 1: A priori estimates on $v(s)$ in $V_A(s)$.}}
\noindent Here, we prepare for the proof of item $(i)$ in Proposition \ref{prop:redu}, which follows if we show that 
$$\left\|\frac{v_-(y,s_1)}{1 + |y|^3}\right\|_{L^\infty} \leq \frac{A}{2s_1^{2 + \eta}} \quad \text{and} \quad \|v_e(s_1)\|_{L^\infty} \leq \frac{A^2}{2s_1^{1/2 + \eta}}.$$
As in \cite{BKnon94} and \cite{MZdm97}, we will make \textit{a priori} estimate on the projections of the Duhamel formulation \eqref{for:vint}, on the negative and exterior part of the solution. The influence of the kernel $\K$ in this formula is very clear. Therefore, it is convenient to give the following result inspired by Bricmont and Kupiainen \cite{BKnon94} which gives the dynamics of the linear operator $\K$:
\begin{lemm}[\textbf{A priori estimates of the linearized operator in the decomposition \eqref{def:decomv}}] \label{lemm:BKespri} For all $\lambda > 0$, there exists $\sigma_0 = \sigma_0(\lambda)$ such that if $\sigma \geq \sigma_0 \geq 1$ and $\vartheta(\sigma)$ satisfies 
\begin{equation}\label{equ:boundpsisigma}
\sum_{m=0}^2|\vartheta_m(\sigma)| + \left\|\frac{\vartheta_-(y,\sigma)}{1+|y|^3}\right\|_{L^\infty} + \|\vartheta_e(\sigma)\|_{L^\infty} < +\infty,
\end{equation}
then, $\theta(s) = \mathcal{K}(s,\sigma)\vartheta(\sigma)$ satisfies for all $s \in [\sigma, \sigma + \lambda]$,
\begin{align}
\left\|\frac{\theta_-(y,s)}{1+|y|^3}\right\|_{L^\infty} &\leq \frac{Ce^{s - \sigma}\left((s - \sigma)^2 + 1\right)}{s}\left(|\vartheta_0(\sigma)| + |\vartheta_1(\sigma)|+\sqrt{s}|\vartheta_2(\sigma)|\right)\nonumber\\
&\qquad + C e^{-\frac{(s - \sigma)}{2}} \left\|\frac{\vartheta_-(y,\sigma)}{1 + |y|^3}\right\|_{L^\infty} + \frac{Ce^{-(s - \sigma)^2}}{s^{3/2}}\|\vartheta_e(\sigma)\|_{L^\infty},\label{equ:boundThe_ne}\\
\|\theta_e(s)\|_{L^\infty} &\leq Ce^{s-\sigma}\left(\sum_{l=0}^2 s^{l/2}|\vartheta_l(\sigma)| + s^{3/2}\left\|\frac{\vartheta_-(y,\sigma)}{1 + |y|^3}\right\|_{L^\infty} \right)\nonumber\\
&\qquad + C e^{-\frac{(s-\sigma)}{p}}\|\vartheta_e(\sigma)\|_{L^\infty}.\label{equ:boundThe_e}
\end{align}
where $C = C(\lambda,K)>0$ ($K$ is given in \eqref{def:chi}), $\vartheta_m,\vartheta_-,\vartheta_e$ and $\theta_m, \theta_-, \theta_e$ are defined by \eqref{de:vbve} and \eqref{def:decomv}.
\end{lemm}
\begin{rema} In view of the formula \eqref{for:vint}, we see that Lemma \ref{lemm:BKespri} will play an important role in deriving the new bounds on the components of $v(s)$ and making our proof simpler. This means that, given bounds on the components of $v(\sigma), B(v(\tau)), R(\tau)$, we directly apply Lemma \ref{lemm:BKespri} with $\K(s, \sigma)$ replaced by $\K(s,\tau)$ and then integrate over $\tau$ to obtain estimates on the components of $v$.
\end{rema}
\begin{rema}\label{rema:proLe} Note that the proof of this result was given by Bricmont and Kupiainen \cite{BKnon94} only when $N=1$ for simplicity. Of course, their proof naturally extends to higher dimensions. Since our paper is relevant only when $N \geq 2$ (otherwise, Fermanian and Zaag proved the result in \cite{FZnon00}  when $N = 1$), we felt we should give the proof of this lemma in higher dimensions for the reader's convenience.
\end{rema}
\begin{proof} Let us mention that Lemma \ref{lemm:BKespri} relays mainly on the understanding of the behavior of the kernel $\mathcal{K}(s,\sigma)$. The proof is essentially the same as in \cite{BKnon94}, but the estimates of those paper did not present explicitly the dependence on all the components of $\vartheta(\sigma)$ which is less convenient for our analysis below. Because the proof is long and technical, we leave it to Appendix \ref{ap:pri}. As we wrote in the remark following Lemma \ref{lemm:BKespri}, we give the proof for all dimensions $N \geq 1$, noting that the proof of \cite{BKnon94} is valid also in all dimensions, though the authors give the proof only when $N = 1$ for simplicity.
\end{proof}

We now assume that for some $\lambda > 0$, for each $s \in [\sigma, \sigma + \lambda]$, we have $v(s) \in V_A(s)$ with $\sigma \geq s_0$. Applying Lemma \ref{lemm:BKespri}, we get new bounds on all terms in the right hand side of \eqref{for:vint}, and then on $v$. More precisely, we claim the following:
\begin{lemm}\label{lemm:estallterms} There exists $A_2 > 0$ such that for each $A \geq A_2$, $\lambda^* > 0$, there exists $\sigma_2(A,\lambda^*) > 0$ with the following property: for all $s_0 \geq \sigma_2(A, \lambda^*)$, for all $\lambda \leq \lambda^*$, assume that for all $s \in [\sigma, \sigma + \lambda]$, $v(s) \in V_A(s)$ with $\sigma \geq s_0$, then there exists $C = C(\lambda^*) > 0$ such that for all $s \in [\sigma, \sigma + \lambda]$,\\
$i)\;$ \textbf{(linear term)}
\begin{align*}
\left\|\frac{\vartheta_-(y,s)}{1 + |y|^3}\right\|_{L^\infty} &\leq \frac{C}{s^{2+\eta}} + \frac{C}{s^{2+\eta}}\left(Ae^{-\frac{s - \sigma}{2}} + A^2e^{-(s - \sigma)^2}\right),\\
\|\vartheta_e(s)\|_{L^\infty} &\leq \frac{C}{s^{1/2+\eta}} + \frac{C}{s^{1/2 + \eta}}\left(A e^{s - \sigma} + A^2e^{-\frac{s-\sigma}{p}} \right),
\end{align*}
where 
$$\K(s,\sigma)v(\sigma) = \vartheta(y,s) = \vartheta_0 + \vartheta_1\cdot y + \frac{1}{2}y^T \vartheta_2 \,y - tr(\vartheta_2) + \vartheta_-(y,s) + \vartheta_e(y,s).$$
If $\sigma = s_0$, we assume in addition that $(d_0, d_1, d_2)$ is chosen so that $(v_0, v_1, v_2)(s_0) \in \hat{V}_A(s_0)$. Then we have for all $s \in [s_0, s_0 + \lambda]$,
$$\left\|\frac{\vartheta_-(y,s)}{1 + |y|^3}\right\|_{L^\infty} \leq \frac{C}{s^{2+\eta}}, \quad \|\vartheta_e(s)\|_{L^\infty} \leq \frac{Ce^{s - s_0}}{s^{1/2 + \eta}}.$$ 
$ii)\,$ \textbf{(remaining terms)}
\begin{align*}
\left\|\frac{\beta_-(y,s)}{1 + |y|^3}\right\|_{L^\infty} \leq \frac{C}{s^{2+\eta}}, \quad \|\beta_e(s)\|_{L^\infty} \leq \frac{C}{s^{1/2 + \eta}},
\end{align*}
where 
\begin{align*}
&\int_\sigma^s\K(s,\tau)\left[B(v(\tau)) + (\gamma(\tau) - \alpha(\tau))v(\tau)\right]d\tau \\
&\qquad \qquad = \beta(y,s) = \beta_0 + \beta_1\cdot y + \frac{1}{2}y^T \beta_2\, y - tr(\beta_2) + \beta_-(y,s) + \beta_e(y,s).
\end{align*}
\end{lemm}
\begin{proof} $i)$ It immediately follows from the definition of $V_A(\sigma)$ and Lemma \ref{lemm:BKespri}. For part $ii)$, all what we need to do is to substitute the estimates on the components of $B(v)$ and 
$$R(y,s) = (\gamma(y,s) - \alpha(y,s))v(y,s)$$
in Lemma \ref{lemm:apestBv} and Lemma \ref{lemm:apestR} into Lemma \ref{lemm:BKespri}, integrating over $[\sigma,s]$ with respect to $\tau$, and taking $\sigma_2(A,\lambda^*)$ large enough, we then have the conclusion. This ends the proof of Lemma \ref{lemm:estallterms}.
\end{proof}

\subsubsection*{\textbf{Step 2: Deriving conclusion $(i)$ of Proposition \ref{prop:redu}.}}
This step is not new and follows also \cite{MZdm97} and \cite{BKnon94}. We give it for the reader's convenience and for the sake of completeness. Here we use Lemma \ref{lemm:estallterms} in order to derive the conclusion of $(i)$ of Proposition \ref{prop:redu}. Indeed, from equation \eqref{for:vint} and Lemma \ref{lemm:estallterms}, we derive new bounds on $\left\|\frac{v_-(y,s)}{1 + |y|^3} \right\|_{L^\infty}$ and $\|v_e(s)\|_{L^\infty}$, assuming that for all $s \in [\sigma, \sigma +\lambda]$, $v(s) \in V_A(s)$, for $\lambda \leq \lambda^*$ and $\sigma \geq s_0 \geq \sigma_1(A,\lambda^*)$ ($\sigma_1$ is given in Lemma \ref{lemm:estallterms}). The key estimate is to show that for $s = \sigma + \lambda$ (or $s \in [\sigma,\sigma +\lambda]$ if $\sigma = s_0$), these bounds are better than those defining $V_A(s)$, provided that $\lambda \leq \lambda^*(A)$. More precisely, we claim the following proposition, which directly yields item $(i)$ of Proposition \ref{prop:redu}:
\begin{prop}[\textbf{Control of $v(s)$ by $(v_0,v_1,v_2)(s)$ in $\hat{V}_A(s)$}] \label{prop:redu1} There exists $A_4 > 1$ such that for each $A \geq A_4$, there exists $\delta_4(A) > 0$ such that for each $s_0 \geq \delta_4(A)$, we have the following properties:\\
- if $(d_0, d_1, d_2)$ is chosen so that $(v_0,v_1,v_2)(s_0) \in \hat{V}_A(s_0)$, and \\
- if for all $s \in [s_0,s_1]$, $v(s) \in V_A(s)$ for some $s_1 \geq s_0$, then: for all $s \in [s_0,s_1]$,
\begin{equation}\label{equ:redu1}
\left\|\frac{v_-(y,s)}{1 + |y|^3} \right\|_{L^\infty}\leq \frac{A}{2s^{2+\eta}}, \quad \|v_e(s)\|_{L^\infty} \leq \frac{A^2}{2s^{1/2 + \eta}}.
\end{equation}
\end{prop}
\noindent Indeed, if $v(s_1) \in \partial V_A(s_1)$, then $v_0(s_1), v_1(s_1), v_2(s_1))$ must be in $\partial \hat{V}_A(s_1)$ from the definition of $V_A(s)$ and \eqref{equ:redu1}. This concludes part $(i)$ of Proposition \ref{prop:redu}, assuming Proposition \ref{prop:redu1} holds.

Let us now give the proof of Proposition \ref{prop:redu1} in order to conclude the proof of part $(i)$ of Proposition \ref{prop:redu}.
\begin{proof}[\textbf{Proof of Proposition \ref{prop:redu1}}] Note that the conclusion of this proposition is very similar to Proposition 3.7, pages 157 in \cite{MZdm97}. But for the reader's convenience, we give here their argument.

Let $\lambda_1 \geq \lambda_2$ be two positive numbers which will be fixed in term of $A$ later. It is enough to show that \eqref{equ:redu1} holds in two cases: $s - s_0 \leq \lambda_1$ and $s -s_0 \geq \lambda_2$. In both cases, we use Lemma \ref{lemm:estallterms} and suppose $A \geq A_2 > 0$, $s_0 \geq \max\{\sigma_2(A,\lambda_1), \sigma_2(A,\lambda_2), \sigma_6(A), 1\}$. \\

\noindent \textbf{Case $s - s_0 \leq \lambda_1$}: Since we have for all $\tau \in [s_0,s]$, $v(\tau) \in V_A(\tau)$, we apply Lemma \ref{lemm:estallterms} with $A$ and $\lambda^* = \lambda_1$, and $\lambda = s - s_0$. From \eqref{for:vint} and Lemma \ref{lemm:estallterms}, we have 
$$\left\|\frac{v_-(y,s)}{1 + |y|^3} \right\|_{L^\infty}\leq \frac{C}{s^{2+\eta}}, \quad \|v_e(s)\|_{L^\infty} \leq \frac{Ce^{\lambda_1}}{s^{1/2 + \eta}}.$$
If we fix $\lambda_1 = \frac{3}{2}\log A$ and $A$ large enough, then \eqref{equ:redu1} satisfies. \\

\noindent \textbf{Case $s- s_0 \geq \lambda_2$}: Since we have for all $\tau \in [\sigma, s]$, $v(\tau) \in V_A(\tau)$, we apply Lemma \ref{lemm:estallterms} with $A$, $\lambda = \lambda^* = \lambda_2$, $\sigma = s - \lambda_2$. From \eqref{for:vint} and Lemma \ref{lemm:estallterms}, we have 
\begin{align*}
\left\|\frac{v_-(y,s)}{1 + |y|^3} \right\|_{L^\infty} &\leq \frac{C}{s^{2+\eta}}\left(1 + A e^{-\frac{\lambda_2}{2}} + A^2e^{-\lambda_2^2}\right),\\
\|v_e(s)\|_{L^\infty} & \leq \frac{C}{s^{1/2 + \eta}}\left(1 + Ae^{\lambda_2} + A^2e^{-\frac{\lambda_2}{p}}\right).
\end{align*}
To obtain \eqref{equ:redu1}, it is enough to have $A \geq 4C$ and
\begin{align*}
& C\left(A e^{-\frac{\lambda_2}{2}} + A^2e^{-\lambda_2^2}\right) \leq \frac{A}{4},\\
& C\left(Ae^{\lambda_2} + A^2e^{-\frac{\lambda_2}{p}}\right) \leq \frac{A^2}{4}.
\end{align*}
If we fix $\lambda_2 = \log(A/8C)$ and take $A$ large enough, we then have the conclusion. This completes the proof of Proposition \ref{prop:redu1} and part $i)$ of Proposition \ref{prop:redu} too.
\end{proof}

\subsubsection*{\textbf{Step 3: Deriving conclusion $(ii)$ of Proposition \ref{prop:redu}.}}
We give the proof of $(ii)$ of Proposition \ref{prop:redu} in this step. We aim at proving that when $(v_0(s),v_1(s),v_2(s))$ touches $\partial \hat{V}_A(s)$ at $s = s_1$, it actually leaves $\hat{V}_A$ at $s_1$  for $s_1 \geq s_0$ where $s_0$ will be large enough. In fact, this is a direct consequence of the following lemma:
\begin{lemm}[\textbf{ODE satisfied by the expanding modes}] \label{lemm:evm012} For all $A > 0$, there exists $\sigma_6(A)$ such that for all $s \geq \sigma_6(A)$, $v(s)\in V_A(s)$ implies that 
\begin{align}
&\left|v_0'(s) - v_0(s) \right|\leq \frac{C}{s^3},\label{equ:odev0}\\
\forall i \in \{1, \cdots, N\},\;&\left|v'_{1,i} - \frac{1}{2}\;v_{1,i}(s) \right| \leq \frac{C}{s^3},\label{equ:odev1}\\
\forall i,j \in \{1, \cdots, N\},\;&\left|h_{ij}'(s) + \frac 2s\;h_{ij}(s) \right| \leq \frac{CA}{s^{3+\eta}},\label{equ:odev2}
\end{align}
where $h(s) = v_2(s) - \frac{\mathcal{A}}{s^2}$.
\end{lemm}
\begin{rema} In comparison with \cite{MZdm97}, we have a new estimate, namely \eqref{equ:odev2}, which will be used to prove the outgoing transverse crossing property on $v_{2,ij}$.
\end{rema}

Let us first derive the conclusion $(ii)$ of Proposition \ref{prop:redu} from Lemma \ref{lemm:evm012}, then we will prove it later. From item $(i)$ of Proposition \ref{prop:redu}, we know that  
\begin{equation}\label{equ:v012tmpode}
v_0(s_1) = \frac{\epsilon A}{s_1^{2+\eta}}, \quad v_{1,i}(s_1) = \frac{\epsilon A}{s_1^{2+\eta}} \quad \text{or}\quad h_{ij}(s_1) = \frac{\epsilon A^2}{s_1^{2 + \eta}},
\end{equation}
for some $\epsilon \in \{-1,1\}$ and $i,j \in \{1,\cdots, N\}$. In order to show that $(v_0(s), v_1(s), v_2(s))$ leaves $\hat{V}_A(s)$ at $s_1$ for $s_1 \geq s_0$, it is enough to show that if \eqref{equ:v012tmpode} holds, then (respectively)
\begin{equation}\label{equ:f2co}
\epsilon\frac{d}{ds}v_0(s_1) > \frac{d}{ds}\left(\frac{A}{s^{2+\eta}}\right)(s_1), \quad \epsilon\frac{d}{ds}v_{1,i}(s_1) > \frac{d}{ds}\left(\frac{ A}{s^{2+\eta}}\right)(s_1),
\end{equation}
or 
\begin{equation}\label{equ:convtileode}
\epsilon\frac{d}{ds}h_{ij}(s_1) > \frac{d}{ds}\left(\frac{A^2}{s^{2+ \eta}}\right)(s_1).
\end{equation}

If $v_0(s_1)$ or $v_{1,i}(s_1)$ touches the boundary of the interval, say for example, when $v_{1,i}(s_1) = \frac{A^2}{s_1^{2+\eta}}$, then we write from \eqref{equ:odev1}, 
$$v_{1,i}'(s_1) \geq \frac{1}{2}v_{1,i}(s_1) - \frac{C}{s_1^3} \geq \frac{A/2 - C}{s_1^{2+\eta}} \geq 0 > -\frac{(2+\eta)A}{s_1^{3+\eta}} = \frac{d}{ds} \left(\frac{A}{s^{2+\eta}}\right)(s_1),$$
provided that $A \geq 2C$. Now, if $h_{ij}(s_1) = \frac{A^2}{s_1^{2 + \eta}}$, then we write from \eqref{equ:odev2} 
\begin{align*}
h_{ij}'(s_1) \geq -\frac{2}{s_1}h_{ij}(s_1) - \frac{CA}{s_1^{3+\eta}} &= -\frac{2A^2}{s_1^{3 + \eta}} - \frac{CA}{s_1^{3 + \eta}}\\
& > -\frac{(2 + \eta)A^2}{s_1^{3 + \eta}} = \frac{d}{ds} \left(\frac{A^2}{s^{2+\eta}}\right)(s_1),
\end{align*}
provided that $A \geq \frac{2C}{\eta}$. All the other cases follow similarly. This concludes part $(ii)$ of Proposition \ref{prop:redu}, assuming Lemma \ref{lemm:evm012} holds.\\

\noindent Let us now give the proof of Lemma \ref{lemm:evm012} to complete the proof of Proposition \eqref{prop:redu}.
\begin{proof}[\textbf{Proof of Lemma \ref{lemm:evm012}}] Estimates \eqref{equ:odev0}, \eqref{equ:odev1} and \eqref{equ:odev2} follow in the same way, though \eqref{equ:odev2} is more delicate. Therefore, we only prove \eqref{equ:odev2}, and refer the interested reader to \cite{MZdm97} (precisely in page 158-159) where a proof similar to the proof of \eqref{equ:odev0} and \eqref{equ:odev1} can be found. In order to prove \eqref{equ:odev2}, we consider $A > 0$ and $s > 0$ which will be taken large in the following and assume that $v(s) \in V_{A}(s)$. 

Let us recall from \eqref{equ:v} the equation satisfied by $v$,
\begin{equation}\label{equ:q}
\partial_s v = (\mathcal{L} + \alpha(y,s))v + B(v) + (\gamma(y,s) - \alpha(y,s))v.
\end{equation}
Note that we have no rest term in equation \eqref{equ:q}, since we linearize here around a solution of \eqref{equ:w}, namely $\hat{w}$, unlike the equation of \cite{MZdm97} where the authors linearize equation \eqref{equ:w} around $f\left(\frac{y}{\sqrt{s}}\right)$, which only an approximate solution of \eqref{equ:w}. This absence of rest term in our setting is the key to \eqref{equ:odev2}, which should be viewed as a refined version of the equation satisfied by $v_{2,ij}$ in \cite{MZdm97} reading as
$$ \left|v_{2,ij}'(s) + \frac{2}{s}v_{2,ij}\right| \leq \frac{C}{s^3},$$
(to derive this equation, the reader should repeat steps at pages 158-159 in \cite{MZdm97} with $m = 2$).

Accordingly, we claim that estimate \eqref{equ:odev2} directly follows from the following inequality:
\begin{equation}\label{equ:tmpv2ij}
\forall i,j \in \{1, \cdots, N\}, \; \left|v_{2,ij}'(s) + \frac{2}{s}\;v_{2,ij}(s) \right| \leq \frac{CA}{s^{3+\eta}}.
\end{equation}
Indeed, since $v_{2}(s) = h(s) + \frac{\mathcal{A}}{s^2}$, we directly obtain \eqref{equ:odev2} by a simple substitution.\\

Let us now focus on the derivation of \eqref{equ:tmpv2ij}. We multiply equation \eqref{equ:q} by $\chi(y,s)\mathcal{M}(y)\rho(y)$, where $\rho$ and $\mathcal{M}$ is introduced in \eqref{def:rho} and \eqref{def:M2}, to get
\begin{equation}\label{est:tmpq0}
\int_{\RN} v_s\chi\mathcal{M}\rho dy = \int_{\RN} \left[\mL v + \alpha v + B(v) + (\gamma - \alpha)v\right] \chi\mathcal{M}\rho dy.
\end{equation}
Arguing as in \cite{MZdm97} (see page 158), we derive for $s$ large
\begin{equation}\label{est:tmpq1}
\left|\int_{\RN} v_s\chi\mathcal{M}\rho dy - \frac{d v_2}{ds} \right| + \left|\int_{\RN} \mL v \chi\mathcal{M}\rho dy\right| \leq Ce^{-s}.
\end{equation}
Recall from Lemma \ref{lemm:apestBv} that $|\chi(y,s) B(v(y,s))| \leq C|v(y,s)|^2$. Hence, 
$$\left|\int_{\RN}B(v)\chi\mathcal{M}\rho dy \right| \leq C\int_{\RN}|v|^2 (1+ |y|^2)\rho dy.$$
Since $v(s) \in V_A(s)$, we have by Definition \ref{def:V_A},
\begin{equation}\label{est:vinVA}
\forall y \in \RN, \quad |v(y,s)| \leq \frac{C}{s^2}(1 + |y|^3),
\end{equation}
Hence
\begin{equation} \label{est:tmpq2}
\left|\int_{\RN}B(v)\chi\mathcal{M}\rho dy \right| \leq \frac{C}{s^4} \int_{\RN} (1 + |y|^8) \rho dy \leq \frac{C}{s^4}.
\end{equation}
From the proof of Lemma \ref{lemm:apestR}, we know that 
$$\forall y \in \RN, \quad |(\gamma(y,s) - \alpha(y,s))\chi(y,s)| \leq \frac{C\log s}{s^2}(1 + |y|^3).$$
This estimate together with \eqref{est:vinVA} yields
\begin{equation}\label{est:tmpq211}
\left|\int_{\RN}(\gamma - \alpha)v\chi\mathcal{M}\rho dy \right| \leq \frac{C\log s}{s^4} \int_{\RN} (1 + |y|^8) \rho dy \leq \frac{C\log s}{s^4}.
\end{equation}
From Lemma \ref{lemm:apEstalp}, we write 
$$
\int_{\RN} \alpha v\chi\mathcal{M}\rho dy = -\frac{1}{4s}\int_{\RN}(|y|^2 - 2N)v\chi\mathcal{M}\rho dy + \int_{\RN} \tilde{\alpha} v\chi\mathcal{M}\rho dy,
$$
where 
$$|\tilde{\alpha}(y,s)| \leq \frac{C}{s^2}(|y|^4 + 1), \quad \forall y \in \RN.$$
Using this estimate together with \eqref{est:vinVA}, we derive 
$$\left|\int_{\RN} \tilde{\alpha} v\chi\mathcal{M}\rho dy \right| \leq \frac{C}{s^4}\int_{\RN}(1 + |y|^7) \rho dy \leq \frac{C}{s^4}.$$
It remains to estimate 
$$R(s) = -\frac{1}{4s}\int_{\RN}(|y|^2 - 2N)v\chi\mathcal{M}\rho dy = -\frac{1}{4s}\int_{\RN} \left(\sum_{k=1}^N \phi_2(y_k)\right)v \chi \mathcal{M} \rho dy,$$
where $\phi_2$ is defined in \eqref{def:phik}.\\
Since $v(y,s)\chi(y,s) = v_b(y,s) = \sum_{\ell = 0}^{+\infty}P_\ell(v_b)(y,s)$ from \eqref{de:vbve} and \eqref{def:decomv}, we get
$$R(s) =  -\frac{1}{4s} \sum_{k = 1}^N \sum_{\ell = 0}^{+\infty} \int_{\RN} P_\ell(v_b)(y,s) \phi_2(y_k)\mathcal{M}(y) \rho(y)dy.$$
Note that for  $\ell \geq 5$ and for all $k \in \{1, \cdots, N\}$, 
$$\int_{\RN} P_\ell v_b(y,s) \mathcal{M}_{ij}(y)\phi_2(y_k) \rho(y)dy = 0$$
because of the orthogonality relation \eqref{equ:orthre}. Therefore, 
$$R(s) = -\frac{1}{4s} \sum_{k = 1}^N \sum_{\ell = 0}^4 \int_{\RN} P_\ell(v_b)(y,s) \phi_2(y_k)\mathcal{M}(y) \rho(y)dy.$$
By straightforward computations, we obtain for $\ell = 2$,
$$\sum_{k = 1}^N \int_{\RN} P_2(v_b)(y,s) \phi_2(y_k)\mathcal{M}(y) \rho(y)dy = 8v_2(s).$$
For $\ell = 0, 1, 3, 4$, we see from Definition \ref{def:V_A} that since $v(s) \in V_A(s)$, then $|v_0(s)| + |v_1(s)| + |v_3(s)| + |v_4(s)| \leq \frac{CA}{s^{2 + \eta}}$;  hence,  $\left|\int_{\RN} P_\ell(v_b)(y,s)\phi_2(y_k)\mathcal{M}(y)\rho(y)dy\right| \leq \frac{CA}{s^{2 + \eta}}$. Thus,
\begin{equation}\label{est:tmpq3}
\left|R(s) + \frac{2}{s}v_2(s)\right| \leq \frac{CA}{s^{3 + \eta}}.
\end{equation}
Substituting estimates \eqref{est:tmpq1}, \eqref{est:tmpq2}, \eqref{est:tmpq211} and \eqref{est:tmpq3} into \eqref{est:tmpq0}, we obtain \eqref{equ:tmpv2ij}. This concludes the proof of Lemma \ref{lemm:evm012} and Proposition \ref{prop:redu} as well. 
\end{proof}

\section{Uniform boundedness up to blow-up of the difference between a solution having the stable profile and a particular constructed solution.}\label{sec:3}
\noindent This section is devoted to the proof of Theorem \ref{theo:2} and Corollary \ref{coro:3}.  Clearly, Corollary \ref{coro:3} directly follows form Theorem \ref{theo:2}. Therefore, we only prove Theorem \ref{theo:2}. Our approach is identical to what done in \cite{FZnon00}. Therefore, we shall refer to \cite{FZnon00} for most of the details and only sketch the main steps of the proof. 
\begin{proof}[Proof of Theorem \ref{theo:2}] Consider $u$ in $\mathbb{B}_{0,T}$ ($\mathbb{B}_{0,T}$ has been introduced in Definition \ref{def:setBaT}) and $\hat{u}$ in $\mathbb{B}'_{0,T}$ is the given radially symmetric and decreasing solution to equation \eqref{equ:problem}. We aim in this section at choosing a particular $\mathcal{A} \in \mathbb{M}_N(\R)$ such that the difference $(T-t)^\frac{1}{p-1}|u(x,t) - \bar{u}_{\mathcal{A}}(x,t)|$ reaches significantly small error terms of order $(T-t)^\lambda$ for some $\lambda > 0$, where $\bar{u}_{\mathcal{A}}(x,t) = u_{\mathcal{A}}(x, t  + T_\mathcal{A} - T) \in \mathbb{B}'_{0,T}$ and $u_{\mathcal{A}}(x,t) \in \mathbb{B}'_{0,T_{\mathcal{A}}}$ is the solution to equation \eqref{equ:problem} constructed in Theorem \ref{theo:1}. The proof will be done through the similarity variables setting \eqref{def:simivars} and we proceed in three steps:
\begin{itemize}
\item[-] \textbf{Step 1}: We work in the $L^2_\rho$ space and show that up to a particular choice of $\mathcal{A}$ in $\mathbb{M}_N(\R)$, the difference $(\T[u](y,s) - \T[\bar{u}_{\mathcal{A}}](y,s))$ in $L^2_\rho$ goes to zero exponentially. This yields an estimate on the difference uniformly for $y$ in compact sets and complete the proof of item $(i)$ in Theorem \ref{theo:2}.
\item[-] \textbf{Step 2}:  We extend the previous convergence from compact sets to larger sets $|y|\leq K\sqrt{s}$, \textit{i.e.} the blow-up region where $|x| \leq K\sqrt{(T-t)|\log (T-t)|}$ after the change \eqref{def:simivars}, thanks to the transport effect of the term $-\frac{1}{2}y\cdot \nabla$ in the definition \eqref{def:opL} of $\mathcal{L}$.
\item[-] \textbf{Step 3}: We use the information on the edge of the blow-up region, \textit{i.e.} when $|x| = K\sqrt{(T-t)|\log (T-t)|}$ as initial data to solve the ODE $u' = u^p$, which gives estimates in the outer region where $\epsilon_0 \geq |x| \geq K\sqrt{(T-t)|\log(T-t)|}$ for some $\epsilon_0  > 0$, thanks to a uniform ODE comparison result for equation \eqref{equ:problem}. Then, gathering the previous information, we  obtain the conclusion of Theorem \ref{theo:2}.
\end{itemize}

\noindent \textbf{Step 1: Exponential decay in $L^2_\rho$ of $(\T[u] - \T[\bar{u}_{\mathcal{A}}])$.}

We prove item $(i)$ in Theorem \ref{theo:2} here. Since the formulation is the same as the one done in \cite{FZnon00}, we therefore follow in extent the strategy of \cite{FZnon00} and focus on the novelties. The general idea is that we first find an equivalent of the difference $(\T[u] - \T[\hat{u}])$ in $L^2_\rho$ through the dynamics of the linearized operator $\mathcal{L}$ defined in \eqref{def:opL}, which yields the fact that the mode of the eigenvalue $1 - \frac{k_0}{2}$ of $\mathcal{L}$ for some $k_0 \geq 2$ is dominant. Then, we replace $\hat{u}$ by $\bar{u}_{\mathcal{A}}$ with a particular choice of $\mathcal{A}$ such that the case when the null mode ($k_0 =2$) is dominant is excluded, hence a negative mode ($k_0 \geq 3$) is dominant which yields the exponential decay of the difference in $L^2_\rho$. More precisely, we claim the following proposition:
\begin{prop}[\textbf{Exponent decay of the difference in $L^2_\rho$}] \label{prop:ExpdecayL2rho} Consider $u \in \mathbb{B}_{0,T}$ and $\hat{u} \in \mathbb{B}'_{0,T}$, where $\hat{u}$ is the given radially symmetric and decreasing solution to equation \eqref{equ:problem}, then, there exists a matrix $\mathcal{A} \in \mathbb{M}_N(\R)$ such that 
\begin{equation}\label{est:prop_ExdA}
\|\T[u](s) - \T[\bar{u}_{\mathcal{A}}](s)\|_{L^2_\rho(\RN)} = \mathcal{O}\left(\frac{e^{-s/2}}{s^3}\right) \quad \text{as} \quad s \to +\infty.
\end{equation}
where 
\begin{equation}\label{def:ubarA}
\bar{u}_{\mathcal{A}}(x,t) = u_{\mathcal{A}}(x, t + T_{\mathcal{A}} - T)
\end{equation}
and $u_{\mathcal{A}}(x,t) \in \mathbb{B}'_{0,T_{\mathcal{A}}}$ is the solution to \eqref{equ:problem} constructed in Theorem \ref{theo:1}.
\end{prop}
\begin{proof} Applying Proposition \ref{prop:apMZ00} to $u$ and $\hat{u}$, we have\\
- either there is a matrix $\mathcal{A}(u,\hat{u}) \in \mathbb{M}_N(\R)$, $\mathcal{A}(u, \hat{u}) \ne 0$ such that 
\begin{equation}\label{equ:estTmpFZ1}
\T[u](y,s) - \T[\hat{u}](y,s) = \frac{1}{s^2}\left(\frac{1}{2}y^T \mathcal{A}y - tr (\mathcal{A}) \right) + o\left(\frac{1}{s^2}\right) \quad \text{in} \; L^2_\rho \;\;\text{as} \;\; s \to +\infty,
\end{equation}
- or there is a constant $C > 0$ such that  for $s$ large,
\begin{equation}\label{equ:estTmpFZ2}
\left\|\T[u](s) - \T[\hat{u}](s)\right\|_{L^2_\rho} \leq \frac{Ce^{-s/2}}{s^3}.
\end{equation}
Applying Theorem \ref{theo:1} with $\mathcal{A} = \mathcal{A}(u, \hat{u})$, we get the existence of a solution $u_\mathcal{A} \in \mathbb{B}'_{0,T_{\mathcal{A}}}$ to equation \eqref{equ:problem} such that 
\begin{equation}\label{equ:estConth1}
\T[u_{\mathcal{A}}](y,s) - \T[\hat{u}](y,s) = \frac{1}{s^2}\left(\frac{1}{2}y^T \mathcal{A}y - tr (\mathcal{A}) \right) + o\left(\frac{1}{s^2}\right) \quad \text{in} \; L^2_\rho \;\; \text{as}\;\; s \to +\infty.
\end{equation}
Note that \eqref{equ:estConth1} is also true when replacing $u_\mathcal{A}$ by $\bar{u}_{\mathcal{A}}$ defined in \eqref{def:ubarA} by the translation invariance of equation \eqref{equ:problem}. Thus, we directly obtain from \eqref{equ:estConth1} and \eqref{equ:estTmpFZ1},
\begin{equation}\label{equ:excludedcas2}
\T[u](y,s) - \T[\bar{u}_{\mathcal{A}}](y,s) =  o\left(\frac{1}{s^2}\right) \quad \text{in} \; L^2_\rho \;\; \text{as} \;\; s \to +\infty.
\end{equation}
Since $\bar{u}_{\mathcal{A}} \in \mathbb{B}_{0,T}$, an alternative application of Proposition \ref{prop:apMZ00} to $u$ and $\bar{u}_{\mathcal{A}}$ also yields \eqref{equ:estTmpFZ1} and \eqref{equ:estTmpFZ2} with $\T[\hat{u}]$ replaced by $\T[\bar{u}_{\mathcal{A}}]$. However, the case \eqref{equ:estTmpFZ1} is excluded by \eqref{equ:excludedcas2}. This concludes the proof of Proposition \ref{prop:ExpdecayL2rho}.
\end{proof}
Since standard parabolic regularity estimates show that \eqref{est:prop_ExdA} also holds in $L^\infty(|y| \leq R)$ for any $R >0$, this concludes the proof of item $(i)$ in Theorem \ref{theo:2}.\\

\noindent \textbf{Step 2: $L^\infty$ estimate in the blow-up region $|y| \leq K\sqrt{s}$.}

In this step, we use the $L^2_\rho$ estimate on the difference $(\T[u] - \T[\bar{u}_{\mathcal{A}}])$ given in Proposition \ref{prop:ExpdecayL2rho} to extend the uniform estimate of the difference on compact sets $|y| \leq K$ to larger sets $|y| \leq K\sqrt{|\log (T-t)|(T-t)}$. Our technique is the same as in \cite{FZnon00} where the authors followed the ideas of \cite{MZgfa98} and \cite{VELcpde92} to estimate the effect of the convective term $-\frac{y}{2}\cdot \nabla $ in $L^q_\rho$ spaces with $q > 1$. Therefore, we only sketch the proof and refer to \cite{FZnon00} for details. We claim the following proposition:
\begin{prop}[\textbf{$L^\infty$ estimate of the difference in the blow-up region}]\label{prop:extcon} For all $K > 0$, there exist $s_0'(\mathcal{A}) \in \mathbb{R}$ and $C = C(K) >0$, such that
\begin{itemize}
\item[(i)] For all $s \geq s_0'$ and for all $|y| \leq K\sqrt{s}$,
$$|\T[u](y,s) - \T[\bar{u}_{\mathcal{A}}](y,s)| \leq C \frac{e^{-s/2}}{s^{3/2}}.$$
\item[(ii)] For all $t \in \left[T - e^{-s_0'}, T\right)$ and for all $x \in B(0, K\sqrt{|\log (T-t)|(T-t)})$, 
$$|u(x,t) - \bar{u}_{\mathcal{A}}(x,t)| \leq C(K)\frac{(T-t)^{\frac{1}{2} - \frac{1}{p-1}}}{|\log(T-t)|^{3/2}}.$$
\end{itemize}
\end{prop}
\begin{proof} Part $(ii)$ immediately follows from part $(i)$ by the transformation \eqref{def:simivars}. As for part $(i)$, we introduce
$$g_{\mathcal{A}}(y,s) = \T[u](y,s) - \T[\bar{u}_{\mathcal{A}}](y,s),$$
then, we see from \eqref{equ:w} that $g_{\mathcal{A}}$ (or $g$ for simplicity) solves the following equation:
\begin{equation}\label{equ:gFZ}
\partial_s g = \Delta g - \frac{y}{2}\cdot \nabla g + (1 + \theta(y,s))g, \quad \forall (y,s) \in \RN\times[\hat{s}, +\infty),
\end{equation}
where $\hat{s} = \max\{-\log T, - \log T_{\mathcal{A}}\}$ and 
\begin{equation}\label{def:thetaFZ}
\theta(y,s) = \frac{|\T[u]|^{p-1}\T[u] - |\T[\bar{u}_{\mathcal{A}}]|^{p-1}\T[\bar{u}_{\mathcal{A}}]}{\T[u] - \T[\bar{u}_{\mathcal{A}}]} - \frac{p}{p-1}.
\end{equation}
We claim that the conclusion $(i)$ is a direct consequence of the following lemma:
\begin{lemm}[\textbf{Extension of the convergence from compact sets to sets $|y| \leq K\sqrt{s}$}] \label{lemm:extencon} Consider $g$ a solution to \eqref{equ:gFZ} and assume that $\theta(y,s) \leq \frac{M}{s}$ and $|g(y,s)| \leq M$ for all $(y,s) \in \RN \times [\hat{s}, +\infty)$. Then, for all $s' \geq \hat{s}$ and $s \geq s' + 1$ such that $e^\frac{s - s'}{2} = K\sqrt{s}$, we have
$$\sup_{|y|\leq \frac{K}{4}\sqrt{s}} |g(y,s)| \leq C(M,K)e^{s - s'}\|g(s')\|_{L^2_\rho}.$$
\end{lemm}
\begin{proof} Lemma \ref{lemm:extencon} is a corollary of Proposition 2.1 in \cite{VELcpde92}. It is proved in the course of the proof of Proposition 2.13 in \cite{VELcpde92}. The reader also find its proof in \cite{FZnon00}, pages 1204-1205. 
\end{proof}

Since $\T[u]$ and $\T[\bar{u}_{\mathcal{A}}]$ are bounded, then $\|g(s)\|_{L^\infty} \leq M$. To show that $\theta(y,s) \leq \frac{M}{s}$, we note from the definition \eqref{def:thetaFZ} of $\theta$ that in general, if $u \ne \bar{u}_{\mathcal{A}}$, then 
$$\theta(y,s) = p|\bar{w}(y,s)|^{p-1} - \frac{p}{p-1}, \quad \text{for some} \; \bar{w} \in (\T[u], \T[\bar{u}_{\mathcal{A}}]).$$
The use of Proposition \ref{prop:MZrefi} yields 
$$\theta(y,s) \leq p \left(\kappa + \frac{C}{s}\right)^{p-1} - p \kappa^{p-1} \leq \frac{M}{s}.$$
Therefore, Lemma \ref{lemm:extencon} and Proposition \ref{prop:ExpdecayL2rho} yield for all $|y| \leq K\sqrt{s}$ and $s \geq s_0'$, 
$$\sup_{|y|\leq \frac{K}{4}\sqrt{s}}|g(y,s)| \leq Ce^{s-s_*}\frac{e^{-s_*/2}}{s_*^{3}},$$
where $e^{\frac{s - s_*}{2}} = K\sqrt{s}$. Since $s_* = s - \log (K^2s) \sim s$ as $s \to +\infty$, conclusion $(i)$ follows. This ends the proof of Proposition \ref{prop:extcon}.
\end{proof}

\bigskip
\noindent \textbf{Step 3: Estimates in the original variables $(x,t)$ and conclusion.}

In this step, we use the uniform bound on $u - \bar{u}_{\mathcal{A}}$ in the region\\
$\{(x,t), |x|\leq K\sqrt{|\log(T-t)|(T-t)}\}$ derived in the previous step and a uniform ODE comparison result in order to extend this bound to the region where $\epsilon_0 \geq |x| \geq K\sqrt{|\log(T-t)|(T-t)}$ for some $\epsilon_0 > 0$. For sake of completeness, we recall their result below and kindly refer the reader to \cite{FZnon00} for the details of the proof.
\begin{prop}[\textbf{Estimates in the intermediate region}] \label{prop:estIterr} There exists $\epsilon_0 > 0$ such that for all $x \in B(0, \delta)$ and $t \in [0, T)$, if $K\sqrt{|\log(T-t)|(T-t)} \leq |x| \leq \epsilon_0$, then 
\begin{align*}
|u(x,t) - \bar{u}_{\mathcal{A}}(x,t)| &\leq C(T-\tilde{t})^{\frac{1}{2} - \frac{1}{p-1}}|\log(T- \tilde{t})|^{-\frac{3}{2}}\\
&\leq C |x|^{1 - \frac{2}{p-1}}|\log |x||^{-\left(2 - \frac{1}{p-1} \right)},
\end{align*}
where $\tilde{t} = \tilde{t}(|x|)$ is defined by
$$|x| = K\sqrt{|\log(T-\tilde{t})|(T - \tilde{t})}.$$
\end{prop}
\begin{proof} See pages 1207-1208 in \cite{FZnon00}.
\end{proof}
\noindent Thus, we have from Propositions \ref{prop:extcon} and \ref{prop:estIterr},
\begin{itemize}
\item[-] if $|x| \leq K\sqrt{|\log(T-t)|(T-t)}$, then 
$$|u(x,t) - u_{\mathcal{A}}(x,t+ T_\mathcal{A} - T)| \leq C(K)(T-t)^{\frac{1}{2} - \frac{1}{p-1}}|\log (T-t)|^{-\frac{3}{2}},$$
\item[-] if $\epsilon_0 \geq |x| \geq K\sqrt{|\log(T-t)|(T-t)}$, then 
$$|u(x,t) - u_{\mathcal{A}}(x,t+ T_\mathcal{A} - T)| \leq C(K)|x|^{1 - \frac{2}{p-1}}|\log |x||^{-\left(2 - \frac{1}{p-1} \right)}.$$
\end{itemize}
which follows estimates \eqref{est:theo2p13}. This concludes the proof of Theorem \ref{theo:2} and Corollary \ref{coro:3} as well.
\end{proof}

\appendix
\renewcommand*{\thesection}{\Alph{section}}
\counterwithin{theo}{section}
\section{Some general results on blow-up solutions to equation   \eqref{equ:problem}.}
In this section, we recall some earlier results  and techniques concerning blow-up solutions of equation \eqref{equ:problem}.
\subsection{Existence of symmetric and radially decreasing solutions to equation \eqref{equ:problem} in $\mathbb{B}'_{0,T}$. } \label{ap:B}
We give in this appendix the existence of radially symmetric and decreasing solutions to equation \eqref{equ:problem} in $\mathbb{B}'_{0,T}$. Let us recall the following result from Bricmont and Kupiainen \cite{BKnon94} and Merle and Zaag \cite{MZdm97}:

\medskip

\noindent \textit{There exists $T_0 > 0$ such that for each $T \in (0,T_0]$, there exists $(d_0, d_1) \in \R \times \RN$ such that  equation \eqref{equ:problem} with initial data 
\begin{equation}\label{equ:intialdataMZ}
u_0(x) = T^{-\frac{1}{p-1}}\left\{f(\xi) \left( 1 + \frac{d_0 + d_1\cdot \xi}{p-1 + \frac{(p-1)^2}{4p}|\xi|^2}\right) \right\}, \quad \xi = \frac{x}{\sqrt{T|\log T|}},
\end{equation}
where $f$ is defined in \eqref{def:fk}, has a unique solution $u \in \mathbb{B}'_{0,T}$. Moreover, there exists $A > 0$ such that
\begin{equation}\label{ap:behaEx}
\T[u](s) - \varphi(s) \in \tilde{V}_A(s), \quad \forall s \geq -\log T,
\end{equation}
where $\varphi$ is defined in \eqref{intr:v} and $\tilde{V}_A(s)$ is the set of all functions $r$ in $L^\infty(\RN)$ such that 
\begin{equation}\label{def:V_AMZ}
\begin{array}{ll}
|r_m(s)| \leq As^{-2} \quad m = 0,1, \quad &|r_2(s)| \leq A^2s^{-2}\log(s),\\
|r_-(y,s)| \leq As^{-2}(1 + |y|^3),\quad  &|r_e(y,s)| \leq A^2s^{-1/2},
\end{array}
\end{equation}
where $r$ is expanded as in \eqref{def:decomv}.
}

\medskip

\noindent In view of \eqref{equ:intialdataMZ}, if we take $d_1 = 0$, then the initial data $u_0(x)$ in \eqref{equ:intialdataMZ} is radially symmetric and decreasing, hence the corresponding solution, say $u(d_0)$, has the same symmetric. In fact, the argument of \cite{BKnon94} and \cite{MZdm97} works with only one variable $d_0$, and we get a different version of the result in the setting of radially decreasing solutions, yielding a particular value $d_0 = \hat{d}_0$ such that the corresponding solution $u(\hat{d}_0) = \hat{u}$ satisfies \eqref{ap:behaEx}. In particular, $\hat{u} \in B'_{0,T}$. Note that the result of \cite{BKnon94} is true for all $N \geq 1$ and $p > 1$, since the authors in \cite{BKnon94} work in the $L^\infty$ space, although the proof in \cite{BKnon94} is given only for $N= 1$ for simplicity. Thus, the existence of a radially symmetric and decreasing solution $\hat{u} \in \mathbb{B}'_{0,T}$ with  \eqref{ap:behaEx} satisfied is true for all $N \geq 1$ and $p > 1$.

\subsection{A classification result of the difference of two solutions in $\mathbb{B}_{0,T}$.}\label{sec:apClas}

In this appendix, we recall the classification result of \cite{FZnon00} mentioned in page \pageref{equ:case1FZ} of the introduction. Let us recall their result in the following proposition:
\begin{prop}[\textbf{Classification of the difference of two solutions in $\mathbb{B}_{0,T}$}] \label{prop:apMZ00} Consider $u_i \in \mathbb{B}_{0,T}$, $i = 1,2$, then, two cases arise: \\
- Either there is a matrix $\mathcal{A} = \mathcal{A}(u_1, u_2) \in \mathbb{M}_N(\R)$ ($\mathcal{A} \neq 0$) such that
\begin{equation}\label{equ:case1FZap}
\T[u_1](y,s) - \T[u_2](y,s) = \dfrac{1}{s^2}\left(\frac{1}{2}y^T\mathcal{A}y - tr(\mathcal{A})\right) + o\left(\frac{1}{s^2} \right) \quad \text{in}\;\; L^2_\rho, \; \; \text{as}\;\; s \to +\infty.
\end{equation}
- Or there is a constant $C > 0$ such that for $s$ large, 
\begin{equation}\label{equ:case2FZap}
\|\T[u_1](s) - \T[u_2](s)\|_{L^2_\rho} \leq \dfrac{Ce^{-s/2}}{s^3}.
\end{equation}
\end{prop}
\begin{proof} Let us define 
$$g(y,s) = \T[u_1](y,s) - \T[u_2](y,s)$$
and denote
$$I(s) = \|g(s)\|_{L^2_\rho}, \quad \ell_k(s) = \|P_k(g)(s)\|_{L^2_\rho},$$
where $P_k$ is defined as in \eqref{def:Projector}. Then, we claim the following:
\begin{lemm}[\textbf{Existence of a dominant component}] \label{lemm:apFZcla} For $s$ large enough, we have\\
(a) For $k \in \{0,1\}$, $\; \ell_k(s) = \mathcal{O}\left(\frac{I(s)}{s}\right).$\\
(b) Only two cases may occur:
\begin{itemize}
\item[(i)] There exists $k_0 \in \mathbb{N}$, $k_0 \not \in \{0,1\}$ so that $I(s) \sim \ell_{k_0}(s)$ and 
\begin{equation*}
\forall k \ne k_0, \quad \ell_k(s) = \mathcal{O}\left(\frac{\ell_{k_0}(s)}{s}\right).
\end{equation*}
Moreover, there exist two positive constants $c$ and $C$ such that 
\begin{equation*}
(cs^c)^{-1}e^{\left(1 - \frac{k_0}{2}\right)s} \leq I(s) \leq Cs^Ce^{\left(1 - \frac{k_0}{2}\right)s}.
\end{equation*}
\item[(ii)] For all $k \in \mathbb{N}$, $\;\ell_k(s) = \mathcal{O}\left( \frac{I(s)}{s}\right)$ and there exists $C_k > 0$ such that
\begin{equation*}
I(s) = \mathcal{O}\left(s^{C_k}e^{\left(1 - \frac{k}{2}\right)s}\right).
\end{equation*}
\end{itemize}
\end{lemm}
\begin{proof} See Proposition 2.6, page 1196 in \cite{FZnon00}.
\end{proof}

\bigskip

Let us give the proof of Proposition \ref{prop:apMZ00} from Lemma \ref{lemm:apFZcla}. We first observe that if case $(i)$ occurs with $k_0 \geq 4$ or case $(ii)$ occurs, then we immediately obtain \eqref{equ:case2FZap}. It remains to examine what happens if $(i)$ occurs with $k_0 = 2, 3$. In particular, we have the following (see Proposition 2.9 in \cite{FZnon00}):\\
- If $I(s) \sim \ell_2(s)$, we have 
$$\forall \beta \in \mathbb{N}^N, \; |\beta| = 2, \quad g_{\beta}'(s) = -\frac{2}{s}g_{\beta} + \mathcal{O}\left(\frac{I(s)}{s^{3/2}}\right),$$
where $g_\beta$ is defined in \eqref{def:Phim} (see page 1200 in \cite{FZnon00} where a similar calculation was given for the case $|\beta| = 3$). From Definition \eqref{def:BoT} and \eqref{equ:c1clas}, we note that $I(s) \leq \frac{C\log s}{s^2}$, hence, 
$$\forall \beta \in \mathbb{N}^N, \; |\beta| = 2, \quad g_\beta(s) = \frac{c_\beta}{s^2} + o\left(\frac{1}{s^2}\right) \quad \text{for some $c_\beta \in \R$}.$$
By definition, this yields \eqref{equ:case1FZap}.\\
- If $I(s) \sim \ell_3(s)$, we have
$$\forall \beta \in \mathbb{N}^N, \; |\beta| = 3, \quad g_{\beta}'(s) = -\left(\frac{1}{2} + \frac{3}{s}\right)g_\beta(s) + \mathcal{O}\left(\frac{I(s)}{s^{3/2}}\right),$$ which implies \eqref{equ:case2FZap}. This concludes the proof of Proposition \ref{prop:apMZ00}.
\end{proof}

\subsection{Uniform $L^\infty$ estimates.}
We recall here the refined $L^\infty$ estimates for solutions to equation \eqref{equ:problem} at blow-up from \cite{MZgfa98} and \cite{MZma00}.
\begin{prop}[\textbf{$L^\infty$ estimates for solution to \eqref{equ:problem} at blow-up}]\label{prop:MZrefi} There exist positive constants $C_1$, $C_2$ and $C_3$ such that if $u$ is a solution to \eqref{equ:problem} which blows up in some finite time $T$ at point $x = 0$, then for all $\epsilon > 0$, there exists $s_1(\epsilon)$ such that for all $s \geq s_1(\epsilon)$
\begin{equation}\label{equ:refiL}
\|\T[u](s)\|_{L^\infty} \leq \kappa + \frac{1}{s}\left(\frac{N\kappa}{2p} + \epsilon\right) \quad \text{and} \quad \|\nabla^i \T[u](s)\|_{L^\infty} \leq \frac{C_i}{s^{i/2}},
\end{equation}
for $i = 1, 2, 3$, where $\T[u]$ is defined in \eqref{def:simivars}.
\end{prop}
\begin{proof} The proof of this proposition can be found in \cite{MZgfa98} and \cite{MZma00}.
\end{proof}

\section{A toolbox for the construction proof.}
In this section, we prove some elementary estimates needed for the proof of Theorem \ref{theo:1}.\\

\noindent The following lemme gives some elementary estimates for the potential $\alpha$ given in equation \eqref{equ:v}:
\begin{lemm}[\textbf{Estimates for the potential $\alpha$}]\label{lemm:apEstalp} There exist a constant $C > 0$ and $s_1 > 0$ such that for all $y \in \mathbb{R}^N$ and $s \geq s_1$, 
\begin{itemize}
\item[i)] $\alpha(y,s) \leq \frac{C}{s}, \quad |\alpha(y,s)| \leq \frac{C}{s}( |y|^2 + 1), \quad \left|\alpha(y,s) + \frac{1}{4s}(|y|^2 - 2N)\right| \leq \frac{C}{s^2}(|y|^4 + 1).$
\item[ii)] $\left|\nabla^i \alpha(y,s) \right| \leq \frac{C}{s^{i/2}}, \; i = 0,1,2.$
\end{itemize}
\end{lemm}
\begin{proof} $i)$ From the definition \eqref{def:alpha} of $\alpha$, we get 
$$\alpha(y,s) \leq p(\varphi(0,s)^{p-1} - \kappa^{p-1}) \leq \frac{C}{s},$$ 
which yields the first estimate. For the next estimates, we introduce 
$$W(Z,s) = \alpha(y,s)\quad \text{ with}\quad Z = \frac{|y|^2}{s}.$$ Taylor expansion of $W(Z,s)$ near $Z = 0$ yields
$$W(Z,s) = W(0,s) + Z\frac{\partial W}{\partial Z}(0,s) + \mathcal{O}(Z^2),$$
where $W(0,s) = \frac{N}{2s} + \mathcal{O}\left(\frac{1}{s^{2}}\right)$ and $\frac{\partial W}{\partial Z}(0,s) = -\frac{1}{4} + \mathcal{O}\left(\frac{1}{s}\right)$. Returning to $\alpha$ yields the last two estimate for $Z$ small. Since $\alpha$ is bounded, the result for $Z$ large is trivial.

$ii)$ By introducing $\hat{W}(z,s) = \alpha(y,s)$ with $z = \frac{y}{\sqrt{s}}$, it is enough to bound $|\nabla^i \hat{W}(z,s)|$ for $i = 1, 2$ which follows easily from the following key estimates
$$\nabla f(z) = \frac{-2bz}{(p-1)}f^p(z) \; \text{and}\; |f| \leq |\varphi|\; \text{with}\; b = \frac{(p-1)^2}{4p}.$$ 
This ends the proof of Lemma \ref{lemm:apEstalp}.
\end{proof}
The following lemmas give estimates on the components of the nonlinear term and the corrective term in equation \eqref{equ:v}.
\begin{lemm}[\textbf{(Estimates for $B(v)$)}]\label{lemm:apestBv} For all $A > 1$, there exists $\sigma_3(A)$ such that for all $\tau \geq \sigma_3(A)$, $v(\tau) \in V_A(\tau)$ implies
\begin{equation*}
m = 0,1,2, \; |B_m(\tau)| \leq \frac{C}{\tau^{4}}, \; \left\|\frac{B_-(y,\tau)}{1+|y|^3}\right\|_{L^\infty} \leq \frac{CA^4}{\tau^{5/2 + 2\eta}},\; \|B_e(\tau)\|_{L^\infty} \leq \frac{CA^{2p'}}{\tau^{(1/2 + \eta)p'}},
\end{equation*}
where $p' = \min \{p,2\}$.
\end{lemm}
\begin{proof} The proof follows directly from the definition of $V_A$ and the fact that  
$$|\chi(y,\tau)B(v(y,\tau)| \leq C|v(y,\tau)|^2, \quad |B(v(y,\tau))| \leq C|v(y,\tau)|^{p'}$$
with $p' = \min\{p,2\}$ (see Lemma 3.15 in \cite{MZdm97} for a similar proof of this fact). Indeed, $v(\tau) \in V_A(\tau)$ implies 
\begin{equation*}
\forall y \in \RN, \quad |v(y,\tau)| \leq \frac{CA^2}{\tau^{2 + \eta}}(1 + |y^3|) + \frac{\|\mathcal{A}\|}{\tau^2}(1 + |y|^2).
\end{equation*}
By definition of $B_m(\tau)$, we see that for $m = 0, 1, 2$,
\begin{align*}
\forall \beta \in \mathbb{N}^N, \;|\beta| = m, \quad |B_\beta(\tau)| &= \frac{1}{\|\phi_\beta(y)\|^2_{L^2_\rho}}\left|\int_{\RN}\phi_\beta(y)B(v(y, \tau))\chi(y, \tau)\rho(y)dy \right|\\
&\leq C\left(\frac{A^4}{\tau^{4 + 2\eta}} + \frac{\|\mathcal{A}\|^2}{\tau^4}\right) \leq  \frac{C}{\tau^4},
\end{align*}
for $\tau$ sufficient large. This yields the estimates for $B_m(\tau)$, $m = 0, 1, 2$.\\
As for $B_-(\tau)$, we write 
\begin{align*}
|\chi(y,\tau)B(v(y,\tau)| &\leq C|v(y,\tau)|^2\\
& \leq C \left(\sum_{m = 0}^2|v_m(\tau)|^2 (1 + |y|^2)^2 + |v_-(y,\tau)|^2 + |v_e(y,\tau)|^2\right)\\
&\leq C \left\{\frac{A^4}{\tau^{4 + 2\eta}}(1 + |y|^6) + \frac{\|\mathcal{A}\|^2}{\tau^4}(1 + |y|^4)\right\}\mathbf{1}_{\{|y| \leq 2K\sqrt{s}\}} + \frac{A^4}{\tau^{1 + 2\eta}}\mathbf{1}_{\{|y| \geq K\sqrt{s}\}}\\
& \leq C\left(\frac{A^4}{\tau^{5/2 + 2\eta}} + \frac{\|\mathcal{A}\|^2}{\tau^{7/2}} \right)(1 + |y|^3)\\
& \leq \frac{CA^4}{\tau^{5/2 + 2\eta}}(1 + |y|^3),
\end{align*}
where $\mathbf{1}_{X}$ is the characteristic function of a set $X$.\\
Hence, 
$$|B_-(y, \tau)| \leq |\chi(y,\tau)B(v(y,\tau)| + \sum_{m = 0}^2|B_m(\tau)|(|y|^2 + 1) \leq \frac{CA^4}{\tau^{5/2 + 2\eta}}(1 + |y|^3).$$
Since $|B(v)| \leq C|v|^{p'}$, we have
$$\|B_e(\tau)\|_{L^\infty} \leq \|B(\tau)\|_{L^\infty} \leq C\|v(\tau)\|^{p'}_{L^\infty} \leq \frac{CA^{2p'}}{\tau^{(1/2 + \eta)p'}}.$$
This concludes the proof of Lemma \ref{lemm:apestBv}.
\end{proof}
\begin{lemm}[\textbf{(Estimates for $(\gamma - \alpha)v$)}] \label{lemm:apestR} For all $A > 1$, there exists $\sigma_4(A) > 0$ such that for all $\tau \geq \sigma_4$, $v(\tau) \in V_A(\tau)$ implies
$$
m = 0,1,2,\; |R_m(\tau)| \leq \frac{C\log \tau}{\tau^4}, \; \left\|\frac{R_-(y,\tau)}{1 + |y|^3}\right\|_{L^\infty} \leq \frac{C\log \tau}{\tau^{5/2}}, \quad \|R_e(\tau)\|_{L^\infty} \leq \frac{CA^2}{\tau^{1/2 + \eta}}\left(\frac{1}{\sqrt{\tau}}\right)^{\bar{p}}, 
$$
where $R(y,s) = (\gamma(y,s) - \alpha(y,s))v(y,s)$ and $\bar{p} = \min\{p-1,1\}$.
\end{lemm}
\begin{proof} The proof is similar to the proof of Lemma \ref{lemm:apestBv}. One can remark from the definition of $\gamma$ and $\alpha$ given in \eqref{def:alpha} and \eqref{def:gamma} (respectively) that
$$|(\gamma(y,s) - \alpha(y,s))\chi(y,s)| \leq C|\hat{w}(y,s) - \varphi(y,s)|,$$
and 
$$|(\gamma(y,s) - \alpha(y,s))| \leq C|\hat{w}(y,s) - \varphi(y,s)|^{\bar{p}},$$ 
where $\bar{p} = \min\{p-1,1\}$.\\
Note from Appendix \ref{ap:B} that $\hat{w}(s) - \varphi(s) \in \tilde{V}_A(s)$ which gives 
$$\forall y \in \RN, \;\; |\hat{w}(y,s) - \varphi(y,s)| \leq \frac{C\log s}{s^2}(1 + |y|^3),$$
and 
$$\|\hat{w}(s) - \varphi(s)\|_{L^\infty} \leq \frac{C}{\sqrt{s}},$$
for $s$ large enough. Using these estimates together with the definition of $V_A(s)$ yields the results. This concludes the proof of Lemma \ref{lemm:apestR}.
\end{proof}

\section{Proof of Lemma \ref{lemm:BKespri}.}\label{ap:pri}
In this appendix, we give the proof of Lemma \ref{lemm:BKespri}. The proof follows from the techniques of Bricmont and Kupiainen \cite{BKnon94} with some additional care, since we give the explicit dependence of the bounds in terms of all the components of initial data. As mentioned earlier, the proof relies mainly on the understanding of the behavior of the kernel $\mathcal{K}(s,\sigma,y,x)$ (see \eqref{def:kernel}). This behavior follows from a perturbation method around $e^{(s-\sigma)\mathcal{L}}(y,s)$, where the kernel of $e^{t\mathcal{L}}$ is given by Mehler's formula:
\begin{equation}\label{for:kernalL}
e^{t\mathcal{L}}(y,x) = \frac{e^t}{(4\pi (1 - e^{-t}))^{\frac N2}} \exp\left[-\frac{|ye^{-\frac t2} - x|^2}{4 (1 - e^{-t})}\right].
\end{equation}
By definition \eqref{def:kernel} of $\mathcal{K}$, we use a Feynman-Kac representation for $\mathcal{K}$:
\begin{equation}\label{for:kernelK}
\mathcal{K}(s,\sigma,y,x) = e^{(s-\sigma)\mathcal{L}}(y,x) \int d\mu_{yx}^{s-\sigma}(\omega)e^{\int_0^{s-\sigma}\alpha(\omega(\tau), \sigma +\tau)d\tau},
\end{equation}
where $d\mu_{yx}^{s-\sigma}$ is the oscillator measure on the continuous paths $\omega: [0, s-\sigma] \to \mathbb{R}^N$ with $\omega(0) = x$, $\omega(s-\sigma) = y$, i.e. the Gaussian probability measure with covariance kernel 
\begin{align}
\Gamma(\tau, \tau') &= \omega_0(\tau)\omega_0(\tau')\nonumber\\
&+2 \left(e^{-\frac{1}{2}|\tau - \tau'|} -  e^{-\frac{1}{2}|\tau + \tau'|} + e^{-\frac{1}{2}|2(s - \sigma) + \tau - \tau'|} -  e^{-\frac{1}{2}|2(s - \sigma) - \tau - \tau'|} \right),\label{def:GammaCo}
\end{align}
which yields $\int d\mu_{yx}^{s-\sigma}(\omega) \omega(\tau) = \omega_0(\tau)$, with 
$$\omega_0(\tau) = \left(\sinh((s-\sigma)/2) \right)^{-1}\left(y\sinh(\frac{\tau}{2}) + x\sinh(\frac{s-\sigma -\tau}{2}) \right).$$
In view of \eqref{for:kernelK}, we can consider the expression for $\mathcal{K}$ as a perturbation of $e^{(s-\sigma)\mathcal{L}}$. Since our potential $\alpha$ defined in \eqref{def:alpha} is the same as in \cite{BKnon94}, we recall some basic properties of the kernel $\mathcal{K}$ in the following lemma:

\begin{lemm}\label{lemm:baskenelK} For all $s \geq \sigma \geq \max\{s_1,1\}$ with $s \leq 2\sigma$ and $s_1$ given in Lemma \ref{lemm:apEstalp}, for all $(y,x) \in \mathbb{R}^N$, we have 
\begin{itemize}
\item[a)] $|\K(s,\sigma,y,x)| \leq Ce^{(s - \sigma)\mL}(y,x)$.
\item[b)] $\K(s,\sigma,y,x) = e^{(s - \sigma)\mL}(y,x)\left(1 + P_2(y,x) + P_4(y,x) \right)$, where 
$$|P_2(y,x)| \leq \frac{C(s - \sigma)}{s}(1 + |y| + |x|)^2,$$
$$ \text{and} \;\;|P_4(y,x)| \leq \frac{C(s - \sigma)(1 + s - \sigma)}{s^2} (1 + |y| + |x|)^4.$$
\item[c)] $\|\K(s,\sigma)(1 - \chi)\|_{L^\infty} \leq Ce^{-\frac{(s - \sigma)}{p}}$.
\end{itemize}
\end{lemm}
\begin{proof}$a)$ From the definition \eqref{for:kernelK} of $\K$ and the fact that $\alpha(y,s) \leq \frac{C}{s}$ (see Lemma \ref{lemm:apEstalp}), we have
\begin{align*}
|\mathcal{K}(s,\sigma,y,x)| &\leq e^{(s-\sigma)\mathcal{L}}(y,x) \int d\mu_{yx}^{s-\sigma}(\omega)e^{\int_0^{s-\sigma}C(\sigma +\tau)^{-1}d\tau}\\
&\leq Ce^{(s-\sigma)\mathcal{L}}(y,x) \int d\mu_{yx}^{s-\sigma} (\omega) \leq C e^{(s-\sigma)\mathcal{L}}(y,x),
\end{align*}
since $s \leq 2\sigma$ and $d\mu_{yx}^{s-\sigma}$ is a probability.\\
For parts $b)$ and $c)$, the reader will find its proof in \cite{BKnon94} (see Lemmas 5 and 7). Although those proofs are written in the one-dimensional case, but they also hold in higher dimensional cases. 
\end{proof}

Before going to the proof of Lemma \ref{lemm:BKespri}, we would like to state some basic estimates which will be frequently used in the proof. 
\begin{lemm}\label{lemm:apC2} For $K$ large enough, we have the following estimates:\\
$a)$ For any polynomial $P$, 
\begin{equation}\label{pro:kerL1}
\int P(y)\mathbf{1}_{\{|y| \geq K\sqrt{s}\}} \rho(y)dy \leq C(P)e^{-s}.
\end{equation}
$b)$ Let $r \geq 0$ and $|f(x)| \leq (1 + |x|)^r$, then 
\begin{equation}\label{pro:kerL2}
|(e^{t\mL}f)(y)| \leq C e^t(1 + e^{-\frac t2}|y|)^r,
\end{equation}
\end{lemm}
\begin{proof} $a)$ follows from a direct calculation. $b)$ follows from the explicit expression \eqref{for:kernalL} by a simple change of variables. 
\end{proof}

\noindent Let us now give the proof of Lemma \ref{lemm:BKespri}.
\begin{proof}[\textbf{Proof of Lemma \ref{lemm:BKespri}}] Let us consider $\lambda > 0$, $\sigma_0 \geq \lambda$, $\sigma \geq \sigma_0$ and $\vartheta(\sigma)$ satisfying \eqref{equ:boundpsisigma}. We want to estimate some components of $\theta(y,s) = \K(s,\sigma)\vartheta(\sigma)$ for each $s \in [\sigma, \sigma + \lambda]$. Since $\sigma \geq \sigma_0 \geq \lambda$, we have 
\begin{equation}\label{equ:ressigma}
\sigma \leq s \leq 2\sigma.
\end{equation}
Therefore, up to a multiplying constant, any power of any $\tau \in [\sigma, s]$ will be bounded systematically by the same power of $s$.

\noindent \textbf{- Estimate for $\theta_e$}: By definition, we write
$$\theta_e(y,s) = (1 - \chi(y,s))\K(s,\sigma)\vartheta(\sigma) = (1 - \chi(y,s))\K(s,\sigma) \left(\vartheta_b(\sigma)+ \vartheta_e(\sigma)\right).$$
Using $c)$ of Lemma \ref{lemm:baskenelK}, we have
$$\|(1 - \chi(y,s))\K(s,\sigma)\vartheta_e(\sigma)\|_{L^\infty} \leq Ce^{-\frac{s - \sigma}{p}}\|\vartheta_e(\sigma)\|_{L^\infty}.$$
It remains to bound $(1 - \chi(y,s))\K(s,\sigma) \vartheta_b(\sigma)$. To this end, we write 
$$\vartheta_b(x,\sigma) = \vartheta_0(\sigma) + \vartheta_1(\sigma)\cdot x + \frac{1}{2}x^T \vartheta_2(\sigma)x - tr(\vartheta_2(\sigma)) + \frac{\vartheta_-(x,\sigma)}{1 + |x|^3}(1 + |x|^3),$$
then use the fact that $\chi(x,\sigma)|x|^k \leq C\sigma^{k/2} \leq Cs^{k/2}$ for $k \in \mathbb{N}$, and $a)$ of Lemma \ref{lemm:baskenelK} to derive
\begin{align*}
\|(1 - \chi(y,s))\K(s,\sigma) \vartheta_b(x,\sigma)\|_{L^\infty} &\leq Ce^{s - \sigma}\sum_{l=0}^2s^{\frac l2}|\vartheta_l(\sigma)|\\
& + Ce^{s - \sigma}s^{\frac 32}\left\|\frac{\vartheta_-(x,\sigma)}{1 + |x|^3}\right\|_{L^\infty}.
\end{align*}
This yields the bound \eqref{equ:boundThe_e}.\\

\noindent \textbf{- Estimate of $\theta_-$:} By definition and from decomposition \eqref{def:decomv}, we write
\begin{align}
\theta_-(y,s) &= P_-\left[\chi(s)\K(s,\sigma)\vartheta(\sigma)\right] = P_-\left[\chi(s)\K(s,\sigma) \left(\vartheta_0(\sigma) + \sum_{|\beta| = 1}\vartheta_\beta(\sigma)\phi_\beta + \sum_{|\beta| = 2}\vartheta_\beta(\sigma)\phi_\beta \right)\right]\nonumber\\
&+ P_-\left[\chi(s)\K(s,\sigma)\vartheta_-(\sigma)\right] + P_-\left[\chi(s)\K(s,\sigma)\vartheta_e(\sigma)\right]:= I + II + III. \label{bound:theta_}
\end{align}
In order to bound $I$, we write $\K(s,\sigma) = \K(s,\sigma) - e^{(s - \sigma)\mL} + e^{(s - \sigma)\mL}$, then we use the fact that $e^{(s-\sigma)\mL}\phi_\beta = e^{(1 - \frac l2)(s - \sigma)}\phi_\beta$ for all $|\beta| = l$, part $b)$ of Lemma \ref{lemm:baskenelK} and \eqref{pro:kerL2} to derive for $l = 0,1,2$,
\begin{align*}
\forall |\beta| = l, \quad &\left|\chi(s)\left(\K(s,\sigma) - e^{(s - \sigma)(1 - \frac l2)}\right)\phi_\beta \right|  = \left|\chi(s)e^{(s-\sigma)\mL}\left(P_2 + P_4\right)\phi_\beta\right|\\
&\leq \frac{Ce^{s-\sigma}(s-\sigma)}{s}\chi(y,s)\left(1 + |y|\right)^{2+l} + \frac{Ce^{s-\sigma}(s - \sigma)(1 + s - \sigma)}{s^{2}}\chi(y,s)\left(1 + |y|\right)^{4 + l}\\
&\leq \left(\frac{Ce^{s-\sigma}(s-\sigma)}{s^{1 - \frac 12\delta_{2,l}}} + \frac{Ce^{s-\sigma}(s - \sigma)(1 + s - \sigma)}{s^{\frac 32 - \frac l2}}\right)(|y|^3 + 1).
\end{align*}
From the easy-to-check fact that 
\begin{equation}\label{equ:propFP_}
\text{if}\;\; |f(y)| \leq m(1+ |y|^3), \;\;\text{then}\;\; \left|P_-\left[f(y)\right]\right| \leq Cm(1 + |y|^3),
\end{equation}
we obtain for $l = 0, 1, 2$,
\begin{align}
\forall |\beta| = l, \;\; &\left|P_- \left[\chi(s) \left(\K(s,\sigma) - e^{(s - \sigma)(1 - \frac l2)}\right)(\vartheta_\beta(\sigma) \phi_\beta) \right]\right|\nonumber\\
&\leq \left(\frac{Ce^{s-\sigma}(s-\sigma)}{s^{1 - \frac 12\delta_{2,l}}} + \frac{Ce^{s-\sigma}(s - \sigma)(1 + s - \sigma)}{s^{\frac 32 - \frac l2}}\right)|\vartheta_\beta(\sigma)|(|y|^3 + 1).\label{equ:estP-1}
\end{align}
Note that $P_-(\phi_\beta) = 0$ for all $|\beta| \leq 2$ and that $|(1 - \chi(y,s))\phi_\beta(y)| \leq Cs^{-\frac 32 + \frac l2}(1 + |y|^3)$. Therefore, we have for $l = 0, 1, 2,$
\begin{align}
\forall |\beta| = l, \;\;\left| P_-\left[\chi(s)e^{(s - \sigma)\mL}(\vartheta_\beta(\sigma)\phi_\beta)\right]\right| &= \left| P_-\left[\vartheta_\beta(\sigma) e^{(s - \sigma)(1 - l/2)}\chi(s)\phi_\beta\right]\right|\nonumber\\
& = \left|\vartheta_\beta(\sigma)e^{(s-\sigma)(1 - \frac l2)}P_- \left[\chi(s) \phi_\beta\right]\right|\nonumber\\
& = \left|\vartheta_\beta(\sigma)e^{(s-\sigma)(1 - \frac l2)}P_- \left[(1 - \chi(s))\phi_\beta\right]\right|\nonumber\\
&\leq \frac{Ce^{(s-\sigma)(1 - l/2)}}{s^{\frac 32-\frac l2}}|\vartheta_\beta(\sigma)|(1 + |y|^3).\label{equ:estP-2}
\end{align}
Since the estimates \eqref{equ:estP-1} and \eqref{equ:estP-2} hold for all $|\beta| = l$ with $l = 0, 1, 2$, we then obtain
\begin{equation}\label{bound:I}
|I| \leq \frac{Ce^{s - \sigma}\left((s - \sigma)^2 + 1\right)}{s}\left( |\vartheta_0(\sigma)| + |\vartheta_1(\sigma)| + \sqrt{s}|\vartheta_2(\sigma)|\right)(1 + |y|^3).
\end{equation}

\noindent In order to bound $III$, we use part $a)$ of Lemma \ref{lemm:baskenelK} and the definition \eqref{for:kernalL} of $e^{(s - \sigma)\mL}$ to write
\begin{align*}
\left\|\frac{\chi(y,s)\K(s,\sigma) \vartheta_e(x,\sigma)}{1 + |y|^3} \right\|_{L^\infty} &\leq C e^{s - \sigma} \|\vartheta_e(\sigma)\|_{L^\infty}\\
&\sup_{|y| \leq 2K\sqrt{s},|x| \geq K\sqrt{\sigma}}e^{-\frac{1}{2}\frac{\left|ye^{-(s - \sigma)/2} - x\right|^2}{4(1 - e^{-(s - \sigma)})}}(1+|y|^3)^{-1}\\
&\leq \left\{\begin{array}{ll}
Cs^{-\frac 32}\|\vartheta_e(\sigma)\|_{L^\infty} & \text{if}\; s-\sigma \leq s_*\\
C e^{-s}\|\vartheta_e(\sigma)\|_{L^\infty}& \text{if}\; s-\sigma \geq s_*
\end{array} \right.
\end{align*}
for a suitable constant $s_*$. Using \eqref{equ:propFP_}, we then get
\begin{equation}\label{bound:III}
|III| \leq Cs^{-\frac 32}e^{-(s - \sigma)^2}\|\vartheta_e(\sigma)\|_{L^\infty}(1 + |y|^3).
\end{equation}
We still have to consider $II$. We consider two cases: \\
- \textit{Case 1:} $s - \sigma \leq 1$.  We directly get from part $a)$ of Lemma \ref{lemm:baskenelK} and part $b)$ of Lemma \ref{lemm:apC2} the following:
\begin{align}
|\K(s, \sigma)\vartheta_-(\sigma)| &=  \left| \int\K(s, \sigma, y, x) \frac{\vartheta_-(x, \sigma)}{1  + |x|^3} (1 + |x|^3)dx \right|\nonumber\\
& \leq C\left\|\frac{\vartheta_-(x, \sigma)}{1  + |x|^3}\right\|_{L^\infty} \int e^{(s - \sigma)\mathcal{L}}(y, x)(1 + |x|^3)dx\nonumber\\
& \leq C\left\|\frac{\vartheta_-(x, \sigma)}{1  + |x|^3}\right\|_{L^\infty} e^{s - \sigma}(1 + |y|^3)\nonumber\\
& \leq C\left\|\frac{\vartheta_-(x, \sigma)}{1  + |x|^3}\right\|_{L^\infty} e^{-\frac{s - \sigma}{2}}(1 + |y|^3) \quad \text{with}\; \; s - \sigma \leq 1, \label{case1:s_sigmal1}
\end{align}

\noindent - \textit{Case 2:} $s - \sigma \geq 1$. We proceed as in \cite{BKnon94} and write
\begin{equation}\label{equ:tmpKe}
\K(s,\sigma)\vartheta_-(\sigma) = \int dx e^{\frac{|x|^2}{4}}\K(s,\sigma)(\cdot,x)f(x) =\int dx G(\cdot,x) E(\cdot,x)f(x),
\end{equation}
where 
\begin{align}
f(x) &= e^{-\frac{|x|^2}{4}}\vartheta_-(x,\sigma),\label{def:apf}\\
G(y,x) &= \frac{e^{s - \sigma}e^{\frac{|x|^2}{4}}}{\left(4\pi(1 - e^{-(s - \sigma)}\right)^{\frac N2}}e^{-\frac{|ye^{-(s - \sigma)/2} - x|^2}{4(1 - e^{-(s - \sigma)})}},\label{def:apG}\\
E(y,x) &= \int d\mu^{s - \sigma}_{yx}(\omega)e^{\int_0^{s - \sigma}\alpha(\omega(\tau),\sigma +\tau)d\tau}.\label{def:apE}
\end{align}
We claim the following lemma whose proof will be given later:
\begin{lemm}\label{ap:lemmC3} Assume that
\begin{equation}\label{equ:apCong}
\int_{\RN} g(x)dx = 0 \;\; \text{and} \;\; |g(x)| \leq A\frac{(1 + |x|^{q + N -1})}{|x|^{N-1}}e^{-\frac{|x|^2}{4}} \;\; \text{for some}\;\; A > 0, \; q \geq 1.
\end{equation}
Then, we can define $g^{(-1)}: \RN \to \RN$ the "antiderivative" of $g$ such that \\
$(i)\;$ $\text{div}\; g^{(-1)}(x) = g(x)$.\\
$(ii)\,$ $\left|g^{(-1)}(x)\right| \leq CA\dfrac{(1 + |x|^{q+ N - 2})}{|x|^{N-1}}e^{-\frac{|x|^2}{4}}$.
\end{lemm}

\noindent An induction application of Lemma \ref{ap:lemmC3} yields the following corollary:
\begin{coro}\label{ap:coroC4} For $m = 1, 2, 3$, there are $F^{(-m)}$ such that
\begin{align*}
&F^{(-1)}: \RN \to \RN \quad \text{and}\quad  \text{div}\,F^{(-1)}(x) = F^{(0)}(x)\equiv f(x),\\
&F^{(-2)}: \RN \to \RN \times \RN \quad \text{and}\quad  \text{div}\,F^{(-2)}_i(x) = F^{(-1)}_i(x),\;\; \forall i \in \{1, \cdots, N\},\\
&F^{(-3)}: \RN \to \RN \times \RN \times \RN \quad \text{and}\quad  \text{div}\,F^{(-3)}_{i,j}(x) = F^{(-2)}_{i,j}(x),\;\; \forall i,j \in \{1, \cdots, N\},
\end{align*}
and 
\begin{equation}\label{est:apFm}
\left|F^{(-m)}(y)\right| \leq C\left\|\frac{\vartheta_-(y,\sigma)}{1 + |y|^3}\right\|_{L^\infty}\frac{\left(1 + |y|^{N + 2 - m)}\right)}{|y|^{N-1}}e^{-\frac{|y|^2}{4}}.
\end{equation}
\end{coro}
\begin{proof} From \eqref{equ:orthre} and the definition \eqref{def:apf} of $f$, we see that 
\begin{equation}\label{ap:estf=0}
\int_{\RN} x^\beta f(x)dx = 0, \quad \forall\beta \in \mathbb{N}^N, \;|\beta| \leq 2.
\end{equation}
Let us write $f(x) = \frac{\vartheta_-(x,\sigma)}{1 + |x|^3} (1 + |x|^3)e^{-\frac{|x|^2}{4}}$, and note that
$$|f(x)| \leq 2\left\|\frac{\vartheta_-(x, \sigma)}{1 + |x|^3}\right\|_{L^\infty} \frac{\left(1 + |x|^{3 + N - 1}\right)}{|x|^{N - 1}}e^{-\frac{|x|^2}{4}}, \quad \forall x \in \RN.$$
Now, we use \eqref{ap:estf=0} with $\beta = 0$, then apply Lemma \ref{ap:lemmC3} with $g = f$, $A = 2\left\|\frac{\vartheta_-(x,\sigma)}{1 + |x|^3}\right\|_{L^\infty}$ and $q = 3$, we get estimate \eqref{est:apFm} for $F^{(-1)}$. Using again \eqref{ap:estf=0} with $|\beta| = 1$, we find that
$$\forall i \in \{1,\cdots, N\}, \quad \int F^{(-1)}_i(x)dx = 0.$$
For each $i\in \{1,\cdots,N\}$, we apply Lemma \ref{ap:lemmC3} with $g = F^{(-1)}_i$, $A = C\left\|\frac{\vartheta_-(x,\sigma)}{1 + |x|^3}\right\|_{L^\infty}$ and $q = 2$ to define $F^{(-2)}_i: \RN \to \RN$ such that $div\, F_i^{(-2)}(x) = F_i^{(-1)}(x)$, and to get the estimate \eqref{est:apFm} for $F^{(-2)}_i$. Similarly, we can define $F^{(-3)}$ from $F^{(-2)}$ and derive the estimate \eqref{est:apFm} by exploiting \eqref{ap:estf=0} with $|\beta| = 2$ and applying Lemma \ref{ap:lemmC3}. This concludes the proof of Corollary \ref{ap:coroC4}.
\end{proof}

\medskip

\noindent Now, using the integration by parts in \eqref{equ:tmpKe}, we write 
\begin{align}
\K(s,\sigma)\vartheta_-(\sigma) &= -\sum_{i = 1}^N\sum_{j = 1}^N\sum_{k=1}^N \int  \frac{\partial^3}{\partial x_k \partial x_j \partial x_i} G(y,x) E(y,x)F^{(-3)}_{i,j,k}(x)dx\nonumber\\
&-\sum_{i = 1}^N\sum_{j = 1}^N \int \frac{\partial^2}{\partial x_j \partial x_i}G(y,x)\left[ \nabla_x E(y,x)\cdot F^{(-3)}_{i,j}(x)\right]dx\nonumber\\
&+ \sum_{i = 1}^N \int \frac{\partial}{\partial x_i}G(y,x)\left[ \nabla_x E(y,x)\cdot F^{(-2)}_{i}(x)\right]dx \nonumber\\
& - \int G(y,x)\left[ \nabla_x E(y,x)\cdot F^{(-1)}(x)\right]dx.\label{equ:Kssigma}
\end{align}
From the definition \eqref{def:apG} of $G(y,x)$, we have 
\begin{equation}\label{est:apG}
|\nabla^{m}_xG(y,x)| \leq Ce^{-\frac{m(s - \sigma)}{2}}(1 + |x| + |y|)^m e^{\frac{|x|^2}{4}}e^{(s - \sigma)\mL}(y,x), \quad m \leq 3.
\end{equation}
Using the integration by parts formula for Gaussian measures (see pages 171-172 in \cite{GJbook87}), we write 
\begin{align*}
&\nabla_x E(y,x) = \frac{1}{2}\int_0^{s - \sigma}\int_0^{s - \sigma} d\tau d\tau' \nabla_x \Gamma(\tau, \tau')\int d\mu_{yx}^{s - \sigma}(\omega)\nabla_x \alpha(\omega(\tau), \sigma+\tau)\\
&\cdot \nabla_x\alpha(\omega(\tau'), \sigma + \tau')e^{\int_0^{s-\sigma}d\tau''\alpha(\omega(\tau''),\sigma + \tau'')}\\
& +\frac{1}{2}\int_0^{s - \sigma}d\tau \nabla_x\Gamma(\tau,\tau')\int d\mu_{yx}^{s - \sigma}(\omega)\Delta_x\alpha(\omega(\tau), \sigma + \tau)e^{\int_0^{s - \sigma}d\tau''\alpha(\omega(\tau''),\sigma + \tau'')}.
\end{align*}
Recalling from Lemma \ref{lemm:apEstalp} that $\alpha(y,s) \leq \frac{C}{s}$ and $\left|\nabla^i \alpha(y,s)\right| \leq \frac{C}{s^{i/2}}$ for $i = 0, 1, 2$, this yields $\int_0^{s - \sigma}\alpha(\omega(\tau), \sigma + \tau)d\tau \leq C$ since $s \leq 2\sigma$. Because $d\mu^{s-\sigma}_{yx}$ is a probability, we then obtain
$$\int d\mu_{yx}^{s -\sigma}(\omega)e^{\int_0^{s -\sigma} d\tau'' \alpha(\omega(\tau''), \sigma + \tau'') d\tau''} \leq C.$$
Combining this with \eqref{def:GammaCo}, we have
\begin{equation}\label{est:apE}
|E(y,x)| \leq C, \quad |\nabla_xE(y, x)| \leq C\frac{(s - \sigma)(1 + s -\sigma)}{s}(|y| + |x|).
\end{equation}
Substituting \eqref{est:apFm}, \eqref{est:apG} and \eqref{est:apE} into \eqref{equ:Kssigma},  we get
\begin{align*}
&|\K(s,\sigma)\vartheta_-(\sigma)| \left\|\frac{\vartheta_-(y,\sigma)}{1+ |y|^3}\right\|_{L^\infty}^{-1} \leq Ce^{-\frac{3}{2}(s - \sigma)}\int_{\RN}e^{(s - \sigma)\mathcal{L}}(y,x)(1 + |y| +|x|)^3 \left(\frac{1 + |x|^{N-1}}{|x|^{N - 1}}\right)dx\\
&+ C\sum_{m = 0}^2e^{-\frac{m}{2}(s - \sigma)}\frac{(s - \sigma)(1 + s - \sigma)}{s}\int_{\RN}e^{(s - \sigma)\mathcal{L}}(y,x)(1 + |y| +|x|)^{m +1} \left(\frac{1 + |x|^{N +1 - m}}{|x|^{N - 1}}\right)dx,
\end{align*}
where $e^{t\mathcal{L}}(y,x)$ is defined in \eqref{for:kernalL}. Since $\frac{(s - \sigma)(1 + s -\sigma)}{s} \leq e^{-\frac{3}{2}(s - \sigma)}$ for $\sigma$ large, we obtain
\begin{align*}
|\K(s,\sigma)\vartheta_-(\sigma)|&\leq Ce^{-\frac{3(s - \sigma)}{2}}\left\|\frac{\vartheta_-(y,\sigma)}{1+ |y|^3}\right\|_{L^\infty}\int_{\RN}e^{(s - \sigma)\mathcal{L}}(y,x)(1 + |y| + |x|)^3 \left(\frac 1 {|x|^{N-1}} + 1\right)dx\\
& = Ce^{-\frac{3(s - \sigma)}{2}}\left\|\frac{\vartheta_-(y,\sigma)}{1+ |y|^3}\right\|_{L^\infty} \big (I_1 + I_2\big),
\end{align*}
where 
\begin{align*}
I_1 &= \int_{|x| \geq 1}e^{(s - \sigma)\mathcal{L}}(y,x)(1 + |y| + |x|)^3 \left(\frac 1 {|x|^{N-1}} + 1\right)dx\\
&\leq 2\int_{|x| \geq 1}e^{(s - \sigma)\mathcal{L}}(y,x)(1 + |y| + |x|)^3dx \leq Ce^{s - \sigma}(1 + |y|^3) \quad (\text{by \eqref{pro:kerL2}}),
\end{align*}
and (note that we are considering the case $s - \sigma \geq 1$)
\begin{align*}
I_2 &= \int_{|x| \leq 1}e^{(s - \sigma)\mathcal{L}}(y,x)(1 + |y| + |x|)^3 \left(\frac 1 {|x|^{N-1}} + 1\right)dx\\
&\leq C(1 + |y|^3)\int_{|x| \leq 1}e^{(s - \sigma)\mathcal{L}}(y,x)|x|^{1-N}dx\\
& = Ce^{s - \sigma}(1 + |y|^3)\int_{|x|\leq 1}\frac{1}{(4\pi(1 - e^{-(s - \sigma)}))^{N/2}}\exp\left(-\frac{|ye^{-\frac{s-\sigma}{2}} - x|^2}{4(1 - e^{-(s - \sigma)})}\right)|x|^{1 - N}dx\\
&\leq \frac{Ce^{s - \sigma}}{(4\pi (1 - e^{-1}))^{N/2}}(1 + |y|^3)\int_{|x|\leq 1}|x|^{1 - N}dx\\
& = \frac{Ce^{s - \sigma}}{(4\pi (1 - e^{-1}))^{N/2}}(1 + |y|^3)N\omega_N\int_{0}^1r^{N-1} r^{1- N} dr \leq Ce^{s - \sigma}(1 + |y|^3),
\end{align*}
(we used in the last line the change of variable $r = |x|$ and $\omega_{N}$ denotes the volume of the ball of radius 1 in $\RN$). Therefore, for $s - \sigma \geq 1$ and for $\sigma$ large enough, we have
$$|\K(s,\sigma)\vartheta_-(\sigma)| \leq Ce^{-\frac{(s - \sigma)}{2}}\left\|\frac{\vartheta_-(y,\sigma)}{1+ |y|^3}\right\|_{L^\infty} (1 + |y|^3).$$
Note that this estimate also holds when $s - \sigma \leq 1$ as proved in \eqref{case1:s_sigmal1}. Hence, we obtain from \eqref{equ:propFP_},

\begin{equation}\label{bound:II}
|II| \leq Ce^{-\frac{(s - \sigma)}{2}}\left\|\frac{\vartheta_-(y,\sigma)}{1+ |y|^3}\right\|_{L^\infty} (1 + |y|^3).
\end{equation}
Substituting \eqref{bound:I}, \eqref{bound:II} and \eqref{bound:III} into \eqref{bound:theta_}, we get the estimate \eqref{equ:boundThe_ne}. This concludes the proof of Lemma \ref{lemm:BKespri}, assuming Lemma \ref{ap:lemmC3} holds.
\end{proof}

Let us give the proof of Lemma \ref{ap:lemmC3} to complete the proof of \eqref{equ:boundThe_ne} and the proof of Lemma \ref{lemm:BKespri} as well.
\begin{proof}[Proof of Lemma \ref{ap:lemmC3}] We apply the Fourier transform to the aimed identity $g = div\, g^{(-1)}$ on the one hand to find 
\begin{equation}\label{equ:apFouF}
\mathcal{F}(g)(\xi) = \mathcal{F} (\text{div}\, g^{(-1)})(\xi) = -\imath \sum_{k = 1}^N \xi_k \mathcal{F}(g^{(-1)}_k)(\xi).
\end{equation}
On the other hand, we use Taylor expansion to $\mathcal{F}(g)(\xi)$ and note that $\mathcal{F}(g)(0) = 0$ thanks to the first identity of \eqref{equ:apCong} to write
$$\mathcal{F}(g)(\xi) = \sum_{k = 1}^N \xi_k \int_0^1 \frac{\partial}{\partial \xi_k} \mathcal{F}(g)(\tau \xi) d\tau.$$
Since $\xi$ is arbitrary, let us define $g^{(-1)}: \RN \to \RN$ by its Fourier transform as follows:
$$\mathcal{F}(g^{(-1)}_k)(\xi) = \imath \int_0^1 \frac{\partial}{\partial \xi_k} \mathcal{F}(g)(\tau \xi) d\tau$$
and check that it satisfies the desired estimate. By \eqref{equ:apFouF}, it satisfies estimate $(i)$. By the inverse Fourier transform, we obtain the explicit formula for $g^{(-1)}_k$ with $k \in \{1, \cdots, N\}$ as follows:
\begin{align*}
g^{(-1)}_k(y) &= \frac{\imath}{2\pi} \int e^{\imath \xi \cdot y} \left(\int_0^1 \frac{\partial}{\partial \xi_k} \mathcal{F}(g)(\tau \xi) d\tau\right) d\xi\\
&= \frac{\imath}{2\pi} \int_0^1 \left(\int e^{\imath \xi \cdot y} \frac{\partial}{\partial \xi_k} \mathcal{F}(g)(\tau \xi) d\xi\right) d\tau\\
&= \frac{\imath}{2\pi} \int_0^1 \left(\int e^{\imath \xi' \cdot y/\tau} \frac{1}{\tau^N}\frac{\partial}{\partial \xi_k} \mathcal{F}(g)(\xi') d\xi'\right) d\tau\\
& = -\frac{\imath}{2\pi} \int_0^1 \left(\int e^{\imath \xi' \cdot y/\tau} \frac{\imath y_k}{\tau^{N+1}}\mathcal{F}(g)(\xi') d\xi'\right) d\tau = \int_0^1 \frac{y_k}{\tau^{N+1}}g\left(\frac{y}{\tau}\right) d\tau.
\end{align*}
Hence, 
$$g^{(-1)}(y) = \int_0^1 \frac{y}{\tau^{N+1}}g\left(\frac{y}{\tau}\right) d\tau.$$
Using the second identity of \eqref{equ:apCong} and a change of variable, we get
\begin{align*}
|g^{(-1)}(y)| &\leq \frac{A}{|y|^{N-1}}\int_0^1\frac{|y|}{\tau^{2}}\left(1 + \frac{|y|^{q + N - 1}}{\tau^{q+N-1}}\right)e^{-\frac{|y|^2}{4\tau^2}}d\tau\\
& = \frac{2A}{|y|^{N-1}} \int_{\frac{|y|}{2}}^{+\infty} (1 + \eta^{q + N - 1})e^{-\eta^2} d\eta \leq CA \frac{1 + |y|^{q + N -2}}{|y|^{N-1}}e^{-\frac{|y|^2}{4}},
\end{align*}
which concludes part $(ii)$. This finishes the proof of Lemma \ref{ap:lemmC3} and closes the proof of  Lemma \ref{lemm:BKespri}.
\end{proof}

\section*{References}

\def\cprime{$'$}

\vspace*{1cm}
\noindent $\begin{array}{ll}
\textbf{Van Tien Nguyen:} & \text{New York University in Abu Dhabi, ​Departement of Mathematics,}\\
& \text{Computational Research Building A2,}\\
& \text{Saadiyat Island, P.O. Box 129188, Abu Dhabi, UAE.}\\
&\textit{Email: Tien.Nguyen@nyu.edu}\\
&\quad \\
\textbf{Hatem Zaag:} & \text{Universit\'e Paris 13, Institut Galil\'ee, LAGA,}\\
& \text{99 Avenue Jean-Baptiste Cl\'ement,}\\
& \text{93430 Villetaneuse, France.}\\
&\textit{Email: Hatem.Zaag@univ-paris13.fr} 
\end{array}$

\end{document}